\documentclass[a4paper,10pt]{article}

\usepackage[utf8]{inputenc}
\usepackage{authblk}
\usepackage{amsthm,amsmath,amssymb,latexsym,dsfont}

\usepackage{enumerate}
\usepackage{tikz}
\usetikzlibrary{arrows.meta, positioning}

\usepackage{xcolor}

\usepackage{url}

\newtheorem{theorem}{Theorem}[section]
\newtheorem{proposition}[theorem]{Proposition}
\newtheorem{lemma}[theorem]{Lemma}
\newtheorem{corollary}[theorem]{Corollary}
\newtheorem{remark}[theorem]{Remark}
\newtheorem{definition}[theorem]{Definition}

\newtheorem{assumption}[theorem]{Assumption}

\newcommand{\N}{\mathbb{N}}
\newcommand{\Z}{\mathbb{Z}}
\newcommand{\Q}{\mathbb{Q}}
\newcommand{\R}{\mathbb{R}}

\newcommand{\B}{\mathcal{B}}

\newcommand{\F}{\mathcal{F}}
\newcommand{\G}{\mathcal{G}}
\newcommand{\HH}{\mathcal{H}}
\newcommand{\X}{\mathcal{X}}
\newcommand{\Y}{\mathcal{Y}}

\renewcommand{\P}{\mathbb{P}}
\newcommand{\E}{\mathbb{E}}

\newcommand{\law}[1]{\text{Law}(#1)}

\newcommand{\Pas}{\text{a.s.}}
\newcommand{\ind}{\mathds{1}}

\newcommand{\eps}{\varepsilon}

\newcommand{\ga}{\gamma}
\newcommand{\ka}{\kappa}

\newcommand{\cov}{\mathrm{Cov}}
\newcommand{\var}{\mathrm{Var}}
\newcommand{\lin}{\mathrm{Lin}}

\newcommand{\dint}{\mathrm{d}}  
\newcommand{\leb}{\mathrm{Leb}}

\usepackage[margin=30mm]{geometry}

\newcommand{\lfrf}[1]{\left\lfloor #1\right\rfloor}

\newcommand{\ul}[1]{\underline{#1}}

\title{Transition of $\alpha$-mixing in Random Iterations with Applications in Queuing Theory\thanks{This work was supported by the János Bolyai Research Scholarship of the Hungarian Academy of Sciences. Additional support was provided by the National Research, Development and Innovation Office (NKFIH) under the Thematic Excellence Program 2021, National Research Subprogram “Artificial intelligence, large networks, data security: mathematical foundations and applications”, and by Grant K 143529.}}

\author{Attila Lovas}
\affil{HUN-REN Alfr\'ed R\'enyi Institute of Mathematics, Budapest, Hungary}
\affil{Budapest University of Technology and Economics, Budapest, Hungary}

\date{\today}

\begin{document}

\maketitle

\begin{abstract}	

Nonlinear time series models with exogenous regressors are essential in econometrics, queuing theory, and machine learning, though their statistical analysis remains incomplete. Key results, such as the law of large numbers and the functional central limit theorem, are known for weakly dependent variables. We demonstrate the transfer of mixing properties from the exogenous regressor to the response via coupling arguments. Additionally, we study Markov chains in random environments with drift and minorization conditions, even under non-stationary environments with favorable mixing properties, and apply this framework to single-server queuing models.
\end{abstract}

\section*{Introduction}

It is very common in natural and social sciences that for describing the time evolution of certain quantity of interests, researchers build models incorporating input variables not influenced by other variables in the system and on which the output variable depends. Such explicative variables, especially in econometrics literature, are called exogeneous covariates. 
Let $\X$, $\Y$, and $\mathcal{Z}$ be complete and separable metric spaces. The $\X$-valued process $(X_t)_{t\in\N}$ represents the time series of interest and the $\Y$-valued process $(Y_t)_{t\in\Z}$ denotes the exogeneous covariate. We postulate that $(X_t)_{t\in\N}$ satisfies the recursion
\begin{align}\label{eq:iter}
\begin{split}
X_{t+1} &= f(X_{t},Y_{t},\eps_{t+1}),
\end{split}
\end{align}
where $X_0$ is a possibly random initial state, $f:\X\times\Y\times \mathcal{Z}\to\X$ is a measurable function, and $\eps_t\in \mathcal{Z}$, $t\in\N$ represents the noise entering to the system. 

The exploration and analysis of non-linear autoregressive processes of this kind constitute a recent and actively developing area of research. In particular, there is a pronounced surge of interest within the fields of applied statistics and econometrics regarding the investigation of standard time series models that incorporate exogeneous regressors. Notable examples include a novel class of Poisson autoregressive models with exogeneous covariates (PARX) introduced by Agosto et al. \cite{PARX_AGOSTO2016640} for modeling corporate defaults. Additionally, the recent research by Gorgi and Koopman \cite{Gorgi2020} has provided valuable insights on observation-driven models involving beta autoregressive processes with exogeneous factors. 
Furthermore, the theory of non-linear autoregressive processes allows researchers for analyzing
large-scale stochastic optimization algorithms, which play a pivotal role in machine learning applications, see \cite{largescaleML2018,lovas}.

\smallskip
The statistical analysis of general non-linear time series models with exogenous covariates  necessitates the law of large numbers (LLN), central limit theorem (CLT), and others. However, this framework is presently unavailable. Researchers have investigated these models under additional assumptions that facilitate their analysis. The ergodicity of iterations given by \eqref{eq:iter} has been extensively studied under the restrictive assumption that the data $(Y_t)_{t\in\Z}$ and the noise $(\eps_t)_{t\in\N}$ are both i.i.d. and also independent of each other (See, \cite{DiFr99}, \cite{Ios09}, and \cite{stenflo}). In this case, the process $(X_t)_{t\in\N}$ is a Markov chain, and this setting now can be considered to be textbook material. Moving beyond this simplifying yet unrealistic assumption, Debaly and Truquet established general results for getting stationarity, ergodicity and stochastic dependence properties for general nonlinear dynamics defined in terms of iterations of random maps \cite{debaly_truquet_2021}. Additionally, there are earlier contributions that consider more general schemes and investigate them without assuming independence. For instance, in the paper of Borovkov and Foss \cite{borovkov1993stochastically}, Foss and Konstantopoulos \cite{foss2003extended} and also in the monograph of Borovkov \cite{borovkov} such processes are treated under the name ''stochastically recursive sequences''. Among the most recent results, we can mention the paper \cite{gyrfi2023strong} by Gy\"orfi et al. that introduces a novel concept called strong stability and provides sufficient conditions for strong stability of iterations given by \eqref{eq:iter}. Furthermore, new findings related to Langevin-type iterations with dependent noise and multitype branching processes were also established.

\smallskip
Assuming that the noise $(\eps_t)_{t\in\N}$ is i.i.d. and independent of the regressor $\mathbf{Y}:=(Y_t)_{t\in\Z}$, we have
\begin{equation}\label{eq:iter_con_ex}
\P (X_{t}\in B\mid (X_j)_{j<t},\,\mathbf{Y}) = \int_{\mathcal{Z}} \ind_{\left\{f(X_{t-1},Y_{t-1},z)\in B\right\}}\,\nu (\dint z),\,\,t\ge 1,
\end{equation}
where $\nu = \law{\eps_0}$. Clearly, the process $(X_t)_{t\in\N}$ defines a time-inhomogeneous Markov chain conditionally on the exogeneous process $(Y_t)_{t\in\N}$ being interpreted as random environment. This characterization leads us to term this process a Markov chain in a random environment (MCRE). This concept is proved to be a good compromise since, many interesting models can be treated as a MCRE. Furthermore, it is worth noting that the rich theory of general state Markov chains equips us with powerful analytical tools to study and understand these processes in-depth. 
Markov chains in random environments were first studied on countable state spaces in
\cite{cogburn1984ergodic,cogburn1990direct,orey1991markov}. On general state spaces \cite{kifer1,kifer1998limit,seppalainen1994large} investigated their ergodic properties under a rather stringent hypothesis: essentially, the Doeblin condition was assumed (see Chapter 16 of \cite{mt}). Such assumptions are acceptable on compact state spaces but they fail in most models evolving in $\R^{d}$. For non-compact state spaces the results of \cite{stenflo} apply (see also Chapter 3 of \cite{borovkov}) but the system dynamics is assumed to be strictly contracting, which, again, is too stringent for most applications. Markov chains in stationary random environments were first treated on non-compact state spaces under Lyapunov and ``small set''-type conditions in \cite{rasonyi2018} and \cite{lovas}. The former paper was based on the control of the maximal process of the random environment but its techniques worked only assuming that the system dynamics is contractive with respect to a certain Lyapunov function, whatever the random environment is. In \cite{lovas} this decreasing property is required only in an \emph{averaged} sense. This result covers important model classes that none of the previous works could: queuing systems with non-independent service times (or inter-arrival times), linear systems that are stable in the average, and stochastic gradient Langevin dynamics when the data is merely stationary. 
In \cite{Truquet1}, under a notably weaker, yet in certain aspects, optimal form of the Lyapunov and the small set conditions, Truquet showed that for a given strongly stationary process $(Y_t)_{t\in\N}$, there exists a process $(X_t)_{t\in\N}$ satisfying the iteration in \eqref{eq:iter}, and the distribution of the process $(X_t,Y_t)_{t\in\N}$ is unique. Additionally, if the process $(Y_t)_{t\in\N}$ is ergodic, then the process $(X_t,Y_t)_{t\in\N}$ is ergodic as well, hence the strong law of large numbers applies.

As far as we know, there are no known results regarding MCREs when the environment $(Y_t)_{t\in\N}$ is non-stationary. Furthermore, the sequence of iterates $(X_t)_{t\in\N}$ is typically non-stationary even in cases when the environment is stationary but the initial state $x_0\in\X$ is independent of $\sigma (\{\eps_t, Y_t\mid t\in\N\})$. Weak dependence assumptions offer a valuable approach to address this problem while allowing for long-range dependencies to be present. The recent work by Truquet \cite{TRUQUET2023294} directed our attention to the fact that through arguments based on coupling inequalities, it can be established under general conditions that the mixing properties of the process $(Y_t)_{t\in\N}$ are inherited by the iterates $(X_t)_{t\in\N}$. 
Combining this idea with Corollary 2 from Herrndorf's paper \cite{herrndorf1984}, we were able to establish the functional central limit theorem for the stochastic gradient Langevin iteration in cases where the data stream is stationary and exhibits favorable mixing properties \cite{lovasCLT}.

\smallskip
Rosenblatt introduced the alpha-mixing coefficient in 1965, defined the class of strongly mixing processes and proved the central limit theorem for strongly mixing stationary processes \cite{rosenblatt1956central}. Over the past decades, researchers have established numerous strong results for non-stationary mixing processes, including various versions of the law of large numbers and the central limit theorem. The main goal of this paper is to investigate the sequence of iterates $(X_t)_{t\in\N}$ through the transitions of mixing properties, leveraging these established results.

The paper is organized as follows: In the first section, we provide sufficient conditions for a recursion of the form \eqref{eq:iter} to inherit the mixing properties of the process $(Y_t)_{t\in\N}$. Leveraging these conditions along with existing results from the literature on strongly mixing sequences, we prove the strong and $L^1$ law of large numbers for suitable functionals of the process $(X_t)_{t\in\N}$. Furthermore, we also show the possibility of constructing confidence intervals.

The second section focuses on the investigation of MCREs under long-term contractivity and minorization conditions satisfied by models discussed in \cite{lovas}. By establishing a coupling inequality and a moment estimate for such chains, the framework presented in Section 1 becomes directly applicable to these processes. Additionally, in this section, using the Cramér-Rao bound, we prove an inequality for variances of sums crucial for the functional cental limit theorem by Merlev{\`e}de and Peligrad \cite{merlevede2020functional}. To the best of our knowledge, this technique represents a novel contribution to the theory of MCREs.
In the third and final section of the paper, we revisit single-server queuing models discussed in \cite{lovas} and \cite{lovas2021ergodic}, and prove the functional central limit theorem for them.

\bigskip
\noindent
{\bf Notations and conventions.} Let $\R_{+}:=\{x\in\R:\, x\geq 0\}$
and $\N_{+}:=\{n\in\N:\ n\geq 1\}$. Let $(\Omega,\F,\P)$ be a probability space. We denote by $\E[X]$ the expectation of a random variable $X$. For $1\le p<\infty$, $L^p$ is used to denote the usual space of $p$-integrable real-valued random variables and $\Vert X \Vert_p$ stands for the $L^p$-norm of a random variable $X$.

In the sequel, we employ the convention that $\inf \emptyset=\infty$, $\sum_{k}^{l}=0$ and $\prod_{k}^{l}=1$ whenever $k,l\in\Z$, $k>l$. 
Lastly, $\langle \cdot\mid \cdot\rangle$ denotes the standard Euclidean inner product
on finite dimensional vector spaces. For example, on $\R^d$, $\langle x\mid y\rangle = \sum_{i=1}^{d} x_i y_i$.

\section{Transition of mixing properties}\label{sec:mixtrans}

In this section, we study the transition of mixing properties of the covariate process to the response and its immediate consequences under minimal assumptions on the iteration \eqref{eq:iter}. To this end, we first introduce the basic concepts that we use in our analysis. Several notions of mixing exist in the literature. The interested reader should consult the excellent survey by Bradley \cite{Bradley2005}, for example. In our context \emph{$\alpha$-mixing} holds particular importance, therefore let us first recall the key concepts related to this type of mixing. We define the measure of dependence, denoted as $\alpha(\G, \HH)$, for any two sub-$\sigma$-algebras $\G,\HH \subset \F$, using the equation:
\begin{equation}\label{eq:dep}
\alpha (\G,\HH) = \sup\limits_{G\in\G, H\in\HH} \left|\P (G\cap H)-\P (G)\P (H)\right|.
\end{equation}

Furthermore, considering an arbitrary sequence of random variables $(W_t)_{t\in\Z}$, we introduce the $\sigma$-algebras $\F_{t,s}^W:=\sigma \left(W_k,\,t\le k\le s\right)$, where $-\infty\le t\le s\le\infty$.
Additionally, we define the dependence coefficients as follows: 
$$\alpha_j^W (n) = \alpha\left(\F_{-\infty,j}^W,\F_{j+n,\infty}^W\right),\,\,j\in\Z.$$ 
The mixing coefficient of $W$ is $\alpha^W (n)= \sup_{j\in\Z}\alpha_j^W (n)$, $n\ge 1$ which is obviously non-increasing. Note that, for strictly stationary $W$, $\alpha_j^W (n)$ does not depend on $j$, and thus $\alpha^W (n)=\alpha_0^W (n)$. We classify $W$ as strongly mixing if $\lim_{n\to\infty}\alpha^W(n)=0$.
Strongly mixing processes are significant because many fundamental theorems known for sequences of i.i.d. random variables can be appropriately extended to such processes. The primary importance of this lies in enabling the statistical analysis of heterogeneously distributed dependent data even when the exact underlying dynamics generating the time series are unknown. The available theoretical framework includes the strong law of large numbers by McLeish \cite{McLeish}, the $L^1$-law of large numbers by Hansen \cite{hansen2019weak}, and the Central Limit Theorem by White \cite{white2000}, among others. For better clarity, these results are collected in Appendix \ref{sec:mixing_survey}.

\medskip
Now, we turn to the analysis of the iteration \eqref{eq:iter}.
For better readability, we introduced the complete separable metric space $\mathcal{Z}$ in the introduction. However, we can always assume that $\mathcal{Z} = [0,1]$ without loss of generality.
Indeed, since $\mathcal{Z}$ is a complete separable metric space, the Borel isomorphism theorem guarantees the existence of a bi-measurable bijection $\phi: \mathcal{Z} \to A$, where $A \subseteq [0,1]$. Given an iteration of the form \eqref{eq:iter}, we can always define an equivalent iteration:
$$
X_{t+1} = \tilde{f}(X_t, Y_t, \tilde{\eps}_{t+1}),
$$
where $\tilde{\eps}_t = \phi(\eps_t) \in [0,1]$, $t \in \N$, and
$$
\tilde{f} (x,y,u) = f(x,y,\phi^{-1}(u)), \quad x \in \X,\, y \in \Y,\, u \in A.
$$
Since this transformation does not alter the dynamics of the iteration, we conclude that it suffices to assume $\mathcal{Z} = [0,1]$. Therefore, from this point onward, we adopt this assumption.

Let us introduce the notation $\tilde{Y}_n = (Y_n,\eps_{n+1})$, $n\in\N$. Furthermore, for a fixed $x\in\X$, we define the process
\begin{equation}\label{eq:Z}
Z_{s,t}^{x} = \begin{cases}
x, & \text{if}\,\, t\le s \\
f\left(Z_{s,t-1}^{x},Y_{t-1},\eps_t\right), & \text{if}\,\, t>s.
\end{cases}
\end{equation}
Note that, for $s\in\N$ and $t\ge s$, $Z_{s,t}^{X_s}=X_t$.

\begin{definition}\label{def:coupling}	
We say that the iteration $\eqref{eq:iter}$ satisfies the \emph{coupling condition} if for some $x_0\in\X$,
$$
\sup_{j\in\N}\P \left(Z_{j,j+n}^{X_j}\ne Z_{j,j+n}^{x_0}\right)
\to 0,
\,\,
n\to\infty.
$$
\end{definition}

Bounding the mixing coefficient of a process is typically a non-trivial task in general. The following lemma presents an upper bound for the $\alpha$-mixing coefficient of the iterates $(X_n)_{n\in\N}$ given $\alpha^{\tilde{Y}}$. 

\begin{lemma}\label{lem:Mixing}
	Assume that $X_0$ is a random initial state independent of $\sigma (Y_n,\eps_{n+1}\mid n\in\N)$, moreover the iteration \eqref{eq:iter} satisfies the coupling condition with $x_0\in\X$.
	
	Then for $0\le m<n$, we have
	$$
	\alpha^X (n) \le \alpha^{\tilde{Y}} (m+1)+b(n-m),
	$$
	where
	$
	b(n)
	=
	\sup_{j\in\N}\P \left(Z_{j,j+n}^{X_j}\ne Z_{j,j+n}^{x_0}\right).
	$
\end{lemma}
\begin{proof}

Let $j \in \N$ and $n \ge 1$ be arbitrary. Suppose $A \in \F_{0,j}^{X}$ and $B \in \F_{j+n, \infty}^{X}$ are arbitrary events. Then by the definition of the generated $\sigma$-algebra, exist collections of Borel sets
$
(A_k)_{k\le j},\,(B_k)_{k\ge j+n}\subseteq\B (\X)
$
such that		
$$
A=(X_k\in A_k,\,0\le k\le j)
\,\,\text{and}\,\,
B=(X_k\in B_k,\,k\ge j+n).
$$

Let \( 0 \le m < n \) be arbitrary, and introduce the event
$
\tilde{B} = \left( Z_{j+m,k}^{x_0} \in B_k, \, k \ge j+n \right).
$
We can estimate
\begin{equation}\label{eq:covest}
	|P(A \cap B) - P(A)P(B)| = |\text{cov}(\ind_A, \ind_B)| 
	\le 
	|\text{cov}(\ind_A, \ind_{\tilde{B}})| + 
	|\text{cov}(\ind_A, \ind_B - \ind_{\tilde{B}})|.
\end{equation}

Using that $A$ is $\sigma (X_0)\vee\F_{0,j-1}^Y\vee\F_{1,j}^\eps$ and $\tilde{B}$ is $\F_{j+m,\infty}^Y\vee\F_{j+m+1,\infty}^\eps$-measurable, and that $X_0$ is independent of $\sigma (Y_n,\eps_{n+1}\mid n\in\N)$, we have
\begin{align*}
	|\text{cov}(\ind_A, \ind_{\tilde{B}})| 
	&\le 
	\E\left| 
	\E \left[\ind_{A}(\ind_{\tilde{B}}-\P(\tilde{B}))
	\middle|X_0
	\right]
	\right|
	\le 
	\alpha (\F_{0,j-1}^Y\vee\F_{1,j}^\eps,\F_{j+m,\infty}^Y\vee\F_{j+m+1,\infty}^\eps)
	\\
	&=\alpha (\F_{0,j-1}^{\tilde{Y}},\F_{j+m,\infty}^{\tilde{Y}})
	\le  \alpha^{\tilde{Y}} (m+1).
\end{align*}

Observe that on the event $\{X_{j+n}=Z_{j+m, j+n}^{x_0}\}$, we have $\ind_{B} = \ind_{\tilde{B}}$, and thus the second term in \eqref{eq:covest} can be estimated as	
\begin{align*}
|\cov(\ind_{A},\ind_{B}-\ind_{\tilde{B}})|
&=
\left|\E [(\ind_{A}-\P (A))(\ind_{B}-\ind_{\tilde{B}})]\right|
\le
\P \left(Z_{0, j+n}^{X_{0}}\ne Z_{j+m, j+n}^{x_0}\right)
\\
&=\P \left(Z_{j+m, j+n}^{X_{j+m}}\ne Z_{j+m, j+n}^{x_0}\right)
\le b(n-m).
\end{align*}	

Given that $A\in\F_{0,j}^{X}$ and $B\in\F_{j+n,\infty}^{X}$ were arbitrary, we have shown that 
$$
\alpha (\F_{0,j}^{X},\F_{j+n,\infty}^{X})
\le
\alpha^{\tilde{Y}} (m+1)
+ b(n-m).
$$
Since the upper bound is independent of \( j \), taking the supremum over \( j \), we arrive at the desired inequality.
\end{proof}

\begin{remark}
	In the above Lemma, it is enough to prescribe that an appropriate version of $X$ satisfy the coupling condition. The proof of the coupling condition for Markov chains in random environments (See Appendix \ref{ap:CouplingCondition}) hinges on this observation.
\end{remark}

\begin{remark}
The recent work by Truquet \cite{TRUQUET2023294} employs a similar inequality (equation (3) on page 3) to establish the transition of mixing in random iteration. However, that approach differs from ours in two significant aspects. Firstly, the analysis in \cite{TRUQUET2023294} is limited to discrete-valued time series models with strictly stationary exogenous covariates. Secondly, unlike in Lemma \ref{lem:Mixing}, the upper estimate of the strong mixing coefficient in \cite{TRUQUET2023294} incorporates the tail sum of non-coupling probabilities.
\end{remark}

Combining Lemma \ref{lem:Mixing} with the theorems on strongly mixing processes presented in Appendix \ref{sec:mixing_survey}, we obtain the following general umbrella theorem. This result equips us with essential tools for the statistical analysis of time series involving non-stationary exogenous covariates, including versions of the weak and strong law of large numbers and theoretical guarantees for constructing confidence intervals. However, this theorem does not cover all relevant aspects. For strongly mixing processes, there are further results concerning the distribution of extreme values \cite{Welsch1972}, concentration inequalities \cite{merlevede2011concentration}, and the law of the iterated logarithm \cite{rio1995LIL}. Leveraging Lemma \ref{lem:Mixing}, these results can be readily extended to random iterations driven by strongly mixing sequences.
\begin{theorem}\label{thm:GeneralResult}
	Assume that $X_0$ is a random initial state independent of $\sigma (Y_n,\eps_{n+1}\mid n\in\N)$, moreover the iteration \eqref{eq:iter} satisfies the coupling condition with $x_0\in\X$.
	
	Consider a measurable function $\Phi:\X \to \R$ such that $\mu_n = \E\left[\Phi (X_n)\right]$ exists and is finite for all $n \in \N$. Let $b(n)$, $n\in\N$ be as in Lemma \ref{lem:Mixing}, and define
	$$
	S_n = \sum_{k=1}^{n} \left(\Phi (X_k) - \mu_k\right), \quad n \ge 1.
	$$
	Then the following statements hold:
	\begin{enumerate}[A)]
		\item If the process $(\tilde{Y}_n)_{n \in \N}$ is strongly mixing and $\sup_{n \in \N} \E |\Phi (X_n)|^p < \infty$ for some $p > 1$, then the $L^1$ law of large numbers holds:
		$$
		\frac{S_n}{n} \stackrel{L^1}{\to} 0,
		\quad \text{as} \quad n \to \infty.
		$$
		
		\item Assume that there exist constants $c > 0$ and $r > 2$ such that
		$$
		\alpha^{\tilde{Y}}(n) + b(n) \le c\, n^{-\frac{r}{r-2}}.
		$$
		Moreover, suppose that for some $r/2 < p \le r < \infty$, we have $\sup_{n \in \N} \E |\Phi(X_n)|^p < \infty$. Then,
		$$
		\frac{S_n}{n} \stackrel{\Pas}{\to} 0,
		\quad \text{as} \quad n \to \infty.
		$$
			
		\item If for $r>2$, $\sup_{n\in\N}\E |\Phi (X_n)|^r<\infty$ and
		$$
		\sum_{k=0}^{\infty} (k+1)^2\left(\alpha^{\tilde{Y}}(k)^{\frac{r-2}{r+2}} + b(k)^{\frac{r-2}{r+2}}\right)<\infty,
		$$
		then the distributions of $n^{-1/2}S_n$ and $\mathcal{N}(0,\var(n^{-1/2}S_n))$ are weakly approaching, that is for any bounded continuous function $g:\R\to\R$,
		$$
		\E \left[g\left(n^{-1/2}S_n\right)\right] - 
		\int_{\R} g\left(\var (n^{-1/2}S_n)^{1/2}t\right)\frac{1}{\sqrt{2\pi}}e^{-\frac{t^2}{2}}\,\dint t
		\to 0\,\,\text{as}\,\,n\to\infty.
		$$ 
		
		Furthermore, there exists $\sigma>0$ such that for any $a>0$, we have
		$$
		\limsup_{n\to\infty}\P (n^{-1/2}|S_n|\ge a)\le \int_{\R}\ind_{[-a,a]^c}(\sigma t)\frac{1}{\sqrt{2\pi}}e^{-\frac{t^2}{2}}\,\dint t.
		$$
	\end{enumerate}	
\end{theorem}
\begin{proof}
	Define the process $W_n = \Phi(X_n) - \mu_n$, for $n \in \N$. By Lemma \ref{lem:Mixing}, we have
	\begin{equation} \label{eq:central}
		\alpha^W(n) \le \alpha^X(n) \le \alpha^{\tilde{Y}}(\lfloor n/2 \rfloor + 1) + b(n - \lfloor n/2 \rfloor).
	\end{equation}
	
	From inequality \eqref{eq:central}, it follows immediately that if $(\tilde{Y}_n)_{n \in \N}$ is strongly mixing, then so is $(W_n)_{n \in \N}$. Thus, part A) follows directly from Theorem \ref{thm:Hansen} and Remark \ref{rem:Hansen}.
	
	For part B), again using \eqref{eq:central} and applying Remark \ref{rem:McLeish}, we see that the conditions of Theorem \ref{thm:McLeish} are satisfied, from which the result follows immediately.
	
	Finally, under the assumptions of part C), we can verify that
	$$
	\sum_{k=0}^{\infty}(k+1)^2 \alpha^W(k)^{\frac{r-2}{r+2}} < \infty,
	$$
	and hence by Theorem \ref{thm:Ekstrom} and Corollary \ref{cor:ekstrom}, the distributions of $n^{-1/2}S_n$ and $\mathcal{N}(0,\var(n^{-1/2}S_n))$ are weakly approaching. Furthermore, Remark \ref{rem:ekstromremark} guarantees the existence of $\sigma > 0$ such that $\sigma_n := \var(n^{-1/2}S_n)^{1/2} < \sigma$ for all $n \ge 1$. Therefore, for any $a > 0$, we have
	\begin{align*}
		\P(n^{-1/2}|S_n| \ge a) &\le \P(n^{-1/2}|S_n| \ge a) - \int_{\R} \ind_{[-a,a]^c}(\sigma_n t) \frac{1}{\sqrt{2\pi}} e^{-t^2/2} \, \mathrm{d}t \\
		&\quad + \int_{\R} \ind_{[-a,a]^c}(\sigma t) \frac{1}{\sqrt{2\pi}} e^{-t^2/2} \, \dint t.
	\end{align*}
	Taking the upper limit as $n \to \infty$ yields
	$$
	\limsup_{n \to \infty} \P(n^{-1/2}|S_n| \ge a) \le \int_{\R} \ind_{[-a,a]^c}(\sigma t) \frac{1}{\sqrt{2\pi}} e^{-t^2/2} \, \dint t.
	$$
	
\end{proof}

\section{Markov chains in random environments}\label{sec:MCRE}

This section is devoted to study an important class of random iterations incorporating exogeneous covariates, called Markov chains in random environments. For convenience, we adopt the parametric kernel formalism to set Lyapunov and ``small set''-type conditions. Let us introduce $$Q(y,x,B)=\int_{[0,1]}\ind_{\left\{f(x,y,z)\in B\right\}}\,\dint z.$$ 
The function $Q:\X\times\Y\times\B (\X)\to [0,1]$ is a parametric probabilistic kernel, which means:
\begin{enumerate}[i]
	\item For each pair $(y, x) \in \Y \times \X$, the mapping $B \mapsto Q(y, x, B)$ defines a Borel probability measure on the Borel sigma-algebra $\B(\X)$.
	
	\item For any choice of set $B \in \B(\X)$, the mapping $(x, y) \mapsto Q(y, x, B)$ is a measurable function with respect to the product sigma-algebra $\B(\X) \otimes \B(\Y)$.
\end{enumerate}

\begin{definition}\label{def:act}
	Let $P:\X\times\B\to [0,1]$ be a probabilistic kernel. For
	a bounded measurable function $\phi:\X\to\R$, we define
	\begin{equation*}
	[P\phi](x)=\int_\X \phi(z) P(x,\dint z),\,x\in\X.
	\end{equation*}
	This definition makes sense for any non-negative measurable $\phi$, too.
\end{definition}
Consistently with Definition \ref{def:act}, for $y \in \Y$, we write $Q(y)\phi$ to denote the action of the kernel $Q(y,\cdot,\cdot)$ on $\phi$. It is important to note that if $y_l, \ldots, y_{k-1} \in \Y$ with $0 \le l < k$, then the successive application of the kernels is interpreted in the order corresponding to the composition of conditional expectations:
\[
[Q(y_{k-1})\ldots Q(y_l)\phi] = [Q(y_l)[\ldots [Q(y_{k-1})\phi]]].
\]
This convention will be important later in the proof of Lemma \ref{lem:ita}.

We say that $Q$ satisfies the drift (or Lyapunov) condition if there exists a measurable mapping $V:\X\to [0,\infty)$, which we call Lyapunov-function, and measurable functions $\ga,K:\Y\to (0,\infty)$, such that for all $(y,x)\in\Y\times\X$,
\begin{equation}\label{eq:Lyapunov}
[Q(y)V](x):=\int_{\X}V(z)\,Q(y,x,\dint z)\le \ga (y)V(x)+K(y).
\end{equation}
We may, and from now on, we will assume that $K(.)\ge 1$ in the drift condition \eqref{eq:Lyapunov}.

The parametric kernel obeys the minorization condition with $R>0$, if there exists a probability kernel $\ka_R: \Y\times\B (\X)\to [0,1]$
and a measurable function $\beta:[0,\infty)\times\Y\to [0,1)$ such that for all
$(y,x,A)\in\Y\times\stackrel{-1}{V}([0,R])\times\B (\X)$,
\begin{equation}\label{eq:smallset}
Q(y,x,A)\ge (1-\beta (R,y))\ka_R (y,A). 
\end{equation}
The minorization condition stipulates the existence of ``small sets''. Therefore, it is also referred to as a "small set"-type condition.

\medskip
If $\ga,K$ are independent of $y$ and $\ga<1$ then \eqref{eq:Lyapunov} is the standard drift condition for geometrically ergodic Markov chains, see Chapter 15 of \cite{mt}. Ergodic properties of Markov chains in stationary random environments was studied by Lovas and R\'asonyi in \cite{lovas} when $\ga(y)\geq 1$ may well occur but the environment satisfies the following long-term contractivity condition:
\begin{equation}\label{eq:LT}
\limsup_{n\to\infty}\E^{1/n}\left(K(Y_0)\prod_{k=1}^{n}\ga (Y_k)\right) < 1.
\end{equation}
Under the assumption that $\E \left[\log (\ga (Y_0))_+\right]+\E \left[\log (K (Y_0))_+\right]<\infty$ and
\begin{equation}\label{eq:Truq}
	\limsup_{n\to\infty}\prod_{k=1}^{n}\ga (Y_{-k})^{1/n}<1,\,\,\P-\Pas,
\end{equation}
which is notably weaker than \eqref{eq:LT}, in \cite{Truquet1} Truquet proved that there exists a stationary process $((Y_t, X_t^\ast))_{t\in\Z}$ satisfying the iteration \eqref{eq:iter}, and the distribution of this process is unique.
If, in addition, the environment $(Y_t)_{t\in\Z}$ is ergodic, the process $((Y_t,X_t^\ast))_{t\in\Z}$ is also ergodic. As a result, the strong law of large numbers holds for any measurable function $\Psi:\Y\times\X\to\R$ with $\E (|\Psi (Y_0,X_0^\ast)|)<\infty$ i.e.
$$
\frac{1}{n}\sum_{k=1}^{n} \Psi (Y_k, X_k^\ast)\to \E (\Psi (Y_0,X_0^\ast)),\,\,\text{as}
\,\,n\to\infty,\,\,\P-\Pas
$$
In this case, the condition \eqref{eq:Truq} boils down to $\E \left[\log (\ga (Y_0))\right]<0$.
Truquet has also shown that if the iteration \eqref{eq:iter} is initialized with deterministic $x_0\in\X$, denoted as $(X_n^{x_0})_{n\in\N}$, then $\law{X_n^{x_0}}$ converges to
$\law{X_0^\ast}$ in total variation as $n\to\infty$, where $\law{X_0^\ast}$ here denotes the marginal distribution of the stationary solution $((Y_t, X_t^\ast))_{t\in\Z}$. On the other hand, contrarily to \cite{lovas}, Truquet did not provide a rate estimate.

\medskip
Independent and identically distributed sequences of random variables $(Y_n)_{n\in\N}$  satisfy the long-term contractivity condition \eqref{eq:LT} if $\E (\ga (Y_0))<1$. Naturally, the question arises whether the inequality \eqref{eq:LT} still holds when $\E (\ga (Y_0))<1$ and the sequence $\alpha^Y (n)$, $n\in\N$, tends to zero rapidly enough. In his Master's thesis on the stability of general state Markov chains \cite{FelsmannThesis}, D\'aniel Felsmann provided an example of a strongly stationary stochastic process $(Y_n)_{n\in\Z}$ and a function $\ga:\Y\to (0,\infty)$ such that $\E (\ga (Y_0))<1$, yet $\lim_{n\to\infty}\E \left(\prod_{k=1}^{n}\ga (Y_k)\right)$, i.e., the long-term contractivity condition in \eqref{eq:LT} is not satisfied. Since the thesis is available only in Hungarian, the example in question is presented in Appendix \ref{ap:Felsmann}.

\medskip
In the forthcoming discussion, we refrain from assuming stationarity for the environment $(Y_n)_{n\in\N}$. Instead, we regard it purely as a sequence of weakly dependent random variables. Consequently, the anticipation of the existence of limiting distributions, as demonstrated in \cite{lovas}, \cite{lovasCLT}, or \cite{Truquet1}, is not viable. Instead, we employ the methodology delineated in Section \ref{sec:mixtrans} to establish the $L^1$-law of large numbers and the functional central limit theorem.

We impose the following additional assumptions on the environment. In absence of stationarity, it is required that the long-term contractivity condition \eqref{eq:LT} holds uniformly along trajectories. With the second condition, essentially, we stipulate that the minorization coefficient $\beta: [0,\infty)\times\Y\to [0,1)$ appearing in \eqref{eq:smallset} can be substituted with a constant on appropriately chosen level sets of the Lyapunov function.  This latter criterion is satisfied in all applications discussed in \cite{lovas}.

\begin{assumption}\label{as:dynamics}
	We assume that the parametric kernel $Q:\Y\times\X\times\B (\X)\to [0,1]$ satisfies the drift condition \eqref{eq:Lyapunov} with $\ga,K:\Y\to (0,\infty)$ such that 
	$$
	\E \left(K(Y_j)\prod_{k=1}^{n}\ga (Y_{k+j})\right)<\infty,
	\,\,j\in\N,\,\, n\ge 1.
	$$
	Furthermore the following two conditions hold:
	\begin{enumerate}[A)]
		\item  
			$\displaystyle{
		\bar{\ga}:=\limsup_{n\to\infty}\sup_{j\ge -1} \E^{1/n}\left(K(Y_j)\prod_{k=1}^{n}\ga (Y_{k+j})\right)<1},
		$ where we define $K(Y_{-1}):=1$.
		
		\item  For some $0<r<1/\bar{\ga}-1$,
		the parametric kernel $Q$ obeys the minorization condition with
		$$R(y)=\frac{2K(y)}{r \ga (y)},\,\, \text{and}
		\,\,
		\bar{\beta}:=\sup_{y\in\Y}\beta \left(R(y),y\right)<1.
		$$
		
	\end{enumerate}

\end{assumption}

\begin{lemma}\label{lem:PhiVbound}
	Let $(X_n)_{n\in\N}$ be a Markov chain in a random environment $(Y_n)_{n\in\N}$ with parametric kernel $Q:\Y\times\X\times \B (\X)\to [0,1]$ satisfying part A) of Assumption \ref{as:dynamics}. Additionally, let $X_0$ be a random initial state independent of $\sigma (Y_n,\eps_{n+1}\mid n\in\N)$ and such that $\E (V(X_0))<\infty$.
	
	Then for any measurable function $\Phi:\X\to\R$ satisfying
	\begin{equation*}
	|\Phi (x)|^{r}\le C (1+V(x)),\,\, x\in\X.
	\end{equation*}
	for some constants $C > 0$ and $r\in\R$, it follows that
	$$
	\sup_{n\in\N}\E \left(|\Phi (X_n)|^{r}\right)<\infty.
	$$
\end{lemma}
\begin{proof}
Using Lemma \ref{lem:ita} from Appendix \ref{ap:CouplingCondition},we can derive the following estimate:
\begin{align*}
	\E ([Q(Y_{n-1})\ldots Q(Y_0)V](X_0))&\le \E (V(X_0))\E \left(\prod_{k=0}^{n-1}\ga (Y_k)\right)
	+\sum_{k=0}^{n-1}\E \left(K(Y_k)\prod_{j=k+1}^{n-1}\ga (Y_j)\right).
\end{align*}
Given Assumption \ref{as:dynamics}, the Cauchy criterion applies to the sum in the second term. Consequently,
$$
\sup_{n\in\N}	\E ([Q(Y_{n-1})\ldots Q(Y_0)V](X_0))<\infty,
$$
which further implies
\begin{align*}
	\sup_{n\in\N}\E \left(|\Phi (X_n)|^{r}\right)&\le  C (1+\sup_{n\in\N}\E (V(X_n)))=
	C (1+\sup_{n\in\N}\E ([Q(Y_{n-1})\ldots Q(Y_0)V](X_0)))
<\infty.
\end{align*}	
\end{proof}

The following lemma ensures that under the drift and minorization conditions specified in Definition \ref{as:dynamics}, the coupling condition holds for a suitable version of the chain.
\begin{lemma}\label{lem:MCREmixing}
	Let $(X_n)_{n\in\N}$ be a Markov chain in a random environment $(Y_n)_{n\in\N}$ with
	random initial state $X_0$ independent of $\sigma (Y_n,\eps_{n+1}\mid n\in\N)$ such that $\E (V(X_0))<\infty$.
	Then under Assumption \ref{as:dynamics}, the chain $(X_n)_{n\in\N}$ admits a random iteration representation of the form \eqref{eq:iter} with appropriate $f:\X\times \Y \times [0,1]\to\X$ which satisfies the coupling condition. More precisely, there exist $c_1,c_2>0$ constants depending only on $\bar{\ga},\bar{\beta}$ and $r$ such that
	$$
	\sup_{j\in\N}\P \left(Z_{j,j+n}^{X_j}\ne Z_{j,j+n}^{x}\right)
	\le c_1 \left(V(x)+\E(V(X_0))+1\right)e^{-c_2n^{1/2}},
	\,\,
	n\ge N
	$$	
	holds for $x\in\X$ and appropriate $N>0$.
\end{lemma}
\begin{proof}
The proof follows a similar argument as in \cite{lovasCLT}. For detailed steps, see Appendix \ref{ap:CouplingCondition}.
\end{proof}

With the above Lemma and the results obtained for general random iterations in Section \ref{sec:mixtrans} in hand, we are now ready to prove the following theorem.
\begin{theorem}\label{thm:MCRE_LLN_and_weakapproaching}

	Let $(X_n)_{n \in \N}$ be a Markov chain in a random environment $(Y_n)_{n \in \N}$ satisfying Assumption \ref{as:dynamics}, with $X_0$ a random initial state independent of $\sigma(Y_n, \varepsilon_{n+1} \mid n \in \N)$, such that $\E[V(X_0)] < \infty$.
	
	Consider a measurable function $\Phi: \X \to \R$ satisfying
	\begin{equation}\label{eq:Phi}
		|\Phi(x)|^p \le C(1 + V(x)), \quad x \in \X,
	\end{equation}
	for some constants $C > 0$ and $p > 1$, and define
	$$
	S_n := \sum_{k=1}^n \left( \Phi(X_k) -  \E[\Phi(X_k)]\right), \quad n \in \N.
	$$
	
	\medskip
	\noindent
	Under the above conditions, the following statements hold:
	\begin{enumerate}[A)]
		\item If the environment process $(Y_n)_{n \in \N}$ is strongly mixing, then the $L^1$ law of large numbers holds:
		$$
		\frac{S_n}{n} \stackrel{L^1}{\to} 0, \quad \text{as } n \to \infty.
		$$
		
		\item Suppose that for some constants $c > 0$ and $r > 2$, the mixing coefficients satisfy $\alpha^Y(n) \le c\, n^{-\frac{r}{r-2}}$, $n \in \N$, and \eqref{eq:Phi} holds with an exponent $p$ such that $r/2 < p \le r < \infty$. Then,
		$$
		\frac{S_n}{n} \stackrel{\text{a.s.}}{\to} 0, \quad \text{as } n \to \infty.
		$$
		
		\item If \eqref{eq:Phi} holds with $p > 2$ and the mixing coefficients satisfy
		\begin{equation}\label{eq:mix_decrease}
			\sum_{n \in \N} (n+1)^2 \alpha^Y(n)^{\frac{p-2}{p+2}} < \infty,
		\end{equation}
		then the distribution of $S_n/\sqrt{n}$ and $\mathcal{N}(0, \var(n^{-1/2}S_n))$ are weakly approaching. Moreover, there exists $\sigma>0$ such that for any $a > 0$, we have
		$$
		\limsup_{n \to \infty} \P\left(n^{-1/2}|S_n| \ge a\right) \le \int_{\R} \ind_{[-a,a]^c}(\sigma t) \frac{1}{\sqrt{2\pi}} e^{-t^2/2} \, \dint t.
		$$
	\end{enumerate}
\end{theorem}
\begin{proof}
	For any $n, j \in \N$, the sigma-algebras $\F_{0,j}^Y \vee \F_{n+j,\infty}^Y$ and $\F_{1,j}^\eps \vee \F_{n+j,\infty}^\eps$ are independent. Hence, by \cite[Lemma 8 on page 13]{BRADLEY19811}, we obtain
	\begin{align*}
		\alpha\left(
		\F_{0,j}^{\tilde{Y}}, \F_{n+j,\infty}^{\tilde{Y}}
		\right)
		&=
		\alpha\left( \F_{0,j}^Y \vee \F_{1,j}^\eps, \F_{n+j,\infty}^Y \vee \F_{n+j,\infty}^\eps \right) \\
		&\le 
		\alpha\left( \F_{0,j}^Y, \F_{n+j,\infty}^Y \right) + \alpha\left( \F_{1,j}^\eps, \F_{n+j,\infty}^\eps \right),
	\end{align*}
	where $\tilde{Y}_n = (Y_n, \eps_{n+1})$ for $n \in \N$.
	
	On the other hand, since $\F_{1,j}^\eps$ and $\F_{n+j,\infty}^\eps$ are also independent, we have $\alpha\left( \F_{1,j}^\eps, \F_{n+j,\infty}^\eps \right) = 0$. By the definition of the dependence coefficient \eqref{eq:dep}, the reverse inequality holds trivially, and therefore
	$$
	\alpha\left( \F_{0,j}^{\tilde{Y}}, \F_{n+j,\infty}^{\tilde{Y}} \right) = \alpha\left( \F_{0,j}^Y, \F_{n+j,\infty}^Y \right), \quad n, j \in \N.
	$$
	This implies that $\alpha^{\tilde{Y}}(n) = \alpha^Y(n)$ for all $n \in \N$.
	
	\medskip
	Furthermore, by Lemma \ref{lem:PhiVbound}, the bound $\sup_{n \in \N} \E |\Phi(X_n)|^p < \infty$ holds for the same $p > 1$ as in \eqref{eq:Phi}.
	
	\medskip
	Finally, by Lemma \ref{lem:MCREmixing}, the chain $(X_n)_{n \in \N}$ admits a random iteration representation satisfying the coupling condition with $b(n) = O(e^{-c n^{1/2}})$ for some $c > 0$.
	
	\medskip
	Therefore, all conditions of Theorem \ref{thm:GeneralResult} are satisfied, which completes the proof.
	
\end{proof}

Following Lindvall \cite{lindvall2002lectures}, Györfi and Morvai introduced a stronger concept of stability known as \emph{forward coupling} \cite{gyorfi2002}. This notion of stability has proven useful in the study of queuing systems. Specifically, under mild ergodicity assumptions, it was found that waiting times in single-server queuing systems operating in a subcritical regime are forward coupled with a stationary and ergodic sequence (See in Section \ref{sec:queuing}).

\begin{definition}\label{def_fwd_cpl}
	We say that the sequence $(W_n)_{n\in\N}$ is \emph{forward coupled} with the sequence $(W_n')_{n\in\N}$ if there exists an almost surely finite random time $\tau$
	such that
	$$
	W_n=W_n'
	$$
	for $n>\tau$.
\end{definition}

The next theorem revisits the case of a stationary environment, previously studied by Gerencsér \cite{gerencser2022invariant}, Rásonyi \cite{rasonyi2018}, Lovas \cite{lovas}, Truquet \cite{Truquet1}, and others, and establishes novel results.
\begin{theorem}\label{thm:stac_forward_coupling}
Let $(Y_n)_{n\in\Z}$ be a strongly stationary process, and assume that Assumption \ref{as:dynamics} holds. Furthermore, let $X_0$ be a random initial state independent of $\sigma (Y_m,\eps_{n+1}\mid m\in\Z,\,n\in\N)$. Consider the stationary process $(Y_n,X_n^\ast)_{n\in\Z}$ satisfying the iteration \eqref{eq:iter}.
Then, appropriate versions of $(X_n)_{n\in\N}$ and $(X_n^\ast)_{n\in\N}$ are forward coupled. Moreover, the tail probability of the random coupling time $\tau$ satisfies the estimate
$$
\P (\tau> n)\le c_1\left(1+\E(V(X_0))\right)e^{-c_2n^{1/2}}
$$
with constants $c_1,c_2>0$ depending only on $\bar{\beta},\bar{\ga}$ and $r$.
\end{theorem}
\begin{proof}
	The proof follows similar lines as the proof of Lemma \ref{lem:MCREmixing}. For readability, it is provided in Appendix \ref{ap:CouplingCondition}.
\end{proof}

The following important corollary of the above theorem provides an explicit and tractable upper bound for the total variation distance between $\law{X_n}$ and $\law{X_n^\ast}$, $n\in\N$. This bound is significantly sharper than those presented in our earlier paper \cite{lovas}.
\begin{corollary}\label{cor_fwd_dtv}
	Under the conditions of Theorem \ref{thm:stac_forward_coupling}, the following rate estimate holds:
	$$
	\lVert\law{X_n}-\law{X_n^\ast}\rVert_{TV}
	\le
	2 c_1\left(1+\E(V(X_0))\right)e^{-c_2n^{1/2}},
	$$
	with the same constants $c_1$ and $c_2$ as in Theorem \ref{thm:stac_forward_coupling}.
\end{corollary}
\begin{proof}
According to the optimal transportation cost characterization of the total variation distance, we have
\begin{align*}
	\frac{1}{2}\lVert\law{X_n}-\law{X_n^\ast}\rVert_{TV}
	\le
	\inf_{\kappa\in\mathcal{C}(X_n,X_n^\ast)}
	\int_{\X\times\X}\ind_{x\ne y}\,\kappa (\dint x, \dint y),
\end{align*}
where $\mathcal{C}(X_n,X_n^\ast)$ denotes the set of probability measures on $\B(\X\times\X)$ with marginals $\law{X_n}$ and $\law{X_n^\ast}$, respectively. By a slight abuse of notation, we can assume that $(X_n)_{n\in\N}$ and $(X_n^\ast)_{n\in\N}$ are forward coupled. Using Theorem \ref{thm:stac_forward_coupling}, we can estimate the right-hand side further by writing
\begin{equation*}
	\inf_{\kappa\in\mathcal{C}(X_n,X_n^\ast)}
	\int_{\X\times\X}\ind_{x\ne y}\,\kappa (\dint x, \dint y)
	\le \P (X_n\ne X_n^\ast) = \P (\tau>n)
	\le c_1\left(1+\E(V(X_0))\right)e^{-c_2n^{1/2}}
\end{equation*}
which gives the desired inequality.
\end{proof}

In the remaining part of this section, we aim to verify the Central Limit Theorem (CLT) for certain functionals of a Markov chain in a non-stationary random environment. There are multiple approaches to achieve this result. For instance, we can use Corollary \ref{cor:ekstrom}, or the results described in Merlev{\`e}de and Peligrad's recent paper \cite{merlevede2020functional}. The advantage of the latter approach is that it not only yields the usual CLT but also the functional CLT, and the conditions on the decay of the mixing properties are weaker. Regardless of the chosen method, we must verify the equality \eqref{eq:varSn}, which states that the variance of the sum grows at the same rate as the sum of the variances. Using either Remark \ref{rem:ekstromremark} or Lemma \ref{lem:PhiVbound}, it can be easily shown that $\var (S_n)$ grows at most at the same rate as the sum of the variances $\var (\Phi (X_k))$, $k=1,2,\ldots,n$. However, it is far from trivial to establish a lower bound for $\var (S_n)$ to ensure the equality in \eqref{eq:varSn}. For this, we 
follow a new approach employing the law of total variance and the Cramér-Rao bound known from estimation theory and information geometry.
For any fixed $n\ge 1$, we treat the conditional distribution of $\ul{X}:=(X_1,\ldots,X_n)$ given the environment $\ul{Y}:=(Y_0,\ldots,Y_{n-1})$ as a parametric statistical model, where the environment $\ul{Y}$ is considered as the parameter.
For technical reasons, we will restrict our discussion in the remaining part of this section to the case when $\Y = \R^m$ with $m \in \N_{+}$.
\begin{assumption}\label{as:Fisher}
	We assume the existence of a reference Borel measure $\nu$ on $\mathcal{B}(\X)$ such that, for all $(y,x) \in \Y \times \X$, the measure $Q(y,x,\cdot)$ is absolutely continuous with respect to $\nu$. Furthermore, the parametric family of conditional densities
	$$
	\Y\ni y\mapsto p_y(z\mid x)=\frac{\dint Q(y,x,\cdot)}{\dint \nu}(z),\,\, (x,z)\in\X\times\X
	$$
	and the measurable function $\Phi:\X\to\R$, playing the role of an estimator,
	satisfy all the regularity conditions required for the Cramér-Rao inequality (See Theorem 1A on page 147 in \cite{BorovkovCramer} or for a more general version, Corollary 5 in \cite{BercherCramer}).
\end{assumption}

\begin{lemma}\label{lem:VarSn}
	Let $(X_n)_{n\in\N}$ be a Markov chain in a random environment $(Y_n)_{n\in\N}$ with parametric kernel
	$Q:\Y\times\X\times\B (\X)\to [0,1]$ and $\Phi:\X\to\R$ satisfying Assumption \ref{as:Fisher}. Suppose that $X_0$ is a random initial state independent of $\sigma (Y_n,\eps_{n+1}\mid n\in\N)$ with distribution $\law{X_0}$ being absolutely continuous w.r.t. $\nu$, where $\nu$ is as in Assumption \ref{as:Fisher}.
	
	Then, for the variance of the parial sum $S_n=\Phi (X_1)+\ldots+\Phi (X_n)$, we have the lower bound
	$$
	\var (S_n) \ge \sum_{k=0}^{n-1} \E \left[\frac{1}{r (I_k)}\lVert 	\partial_{y_k}\E \left[S_n\mid \ul{Y}\right]\rVert^2\right],\,\,n\ge 1,
	$$
	where $r(I_k)$ denotes the spectral radius of the Fisher information matrix
	$$
	I_k:=
	\E\left[\left(\partial_{y}\log p_y(X_{k+1}\mid X_{k})\right)^\top\left(\partial_{y}\log p_y(X_{k+1}\mid X_{k})\right)\middle| \ul{Y}\right].
	$$
\end{lemma}
\begin{proof}
	By the law of total variance, we can write
	\begin{align}\label{eq:totvar}
	\begin{split}
	\var (S_n) &= \E \left[\var (S_n \mid \F_{0,n-1}^Y)\right] + 
	\var\left(\E \left[S_n\mid\F_{0,n-1}^Y \right]\right)
	\ge \E \left[\var (S_n \mid\F_{0,n-1}^Y)\right].
	\end{split}
	\end{align}
	For typographical reasons, we use the notations $\var_{\ul{Y}}(S_n):=\text{Var} (S_n \mid \ul{Y})$ and $\cov_{\ul{Y}}(\cdot,\cdot):=\cov(\cdot,\cdot\mid\ul{Y})$. 
	The conditional variance of the partial sum can be expressed as follows:
	\begin{equation}\label{eq:varbound}
	\var_{\ul{Y}}(S_n)=\sum_{k,l=1}^{n}\cov_{\ul{Y}} (\Phi (X_k),\Phi (X_l)) = \mathbf{1}^\top\cov_{\ul{Y}} (\Phi (\ul{X}))\mathbf{1},
	\end{equation}
	where $\cov_{\ul{Y}} (\Phi (\ul{X}))$ is the conditional covariance matrix of the random vector 
	$$
	\Phi(\ul{X})=(\Phi (X_1),\ldots,\Phi (X_n))^\top,
	$$ 
	$\mathbf{1} = [1,\ldots,1]^\top\in\R^n$ denotes the size $n$ vector of $1$-s.
		
	Using the Cram\'er-Rao inequality, we arrive at the following estimate
	\begin{equation}\label{eq:CR}
	\cov_{\ul{Y}} (\Phi (\ul{X}))\ge (\partial_{\ul{y}}\phi (\ul{Y}))I(\ul{Y})^{-1}(\partial_{\ul{y}}\phi (\ul{Y}))^\top,
	\end{equation}
	where $\phi:\Y^n\to\R^n$, $\phi(\ul{y})=\E\left[\Phi (\ul{X})\mid \ul{Y}=\ul{y}\right]$ and $I(\ul{y})$ is the Fisher information matrix. 
	
	\smallskip
	Note that $\partial_{\ul{y}}\phi (\ul{y})\in\lin (\Y^n,\R^n)$ and $I(\ul{y})\in\lin (\Y^n)$ are linear operators with block representation
	\begin{align*}
	\partial_{\ul{y}}\phi(\ul{y}) &= \left[\partial_{y_0}\phi(\ul{y}),\ldots,\partial_{y_{n-1}}\phi(\ul{y})\right] \\
	[I(\ul{y})]_{kl} &= \E\left[\left(\partial_{y_k}\log p(\ul{X}\mid\ul{y})\right)^\top\left(\partial_{y_l}\log p(\ul{X}\mid\ul{y})\right)\middle| \ul{Y}=\ul{y}\right] = -\E \left[\partial^2_{y_k y_l}\log p(\ul{X}\mid\ul{y})\middle| \ul{Y}=\ul{y}\right],
	\end{align*}
	where $\partial_{y_i}\phi(\ul{y})\in\lin (\Y,\R^n)$, $i=0,\ldots,n-1$, and $[I(\ul{y})]_{kl}\in\lin (\Y)$, $k,l=0,\ldots,n-1$.
	
	By the Markov property, for the conditional density of $\ul{X}$ given $\ul{y}$
	can be expressed as the product of $\pi (X_0)=\frac{\dint\law{X_0}}{\dint\nu}$ and the parametric transition densities $ p_{y_{k-1}}(x_{k}\mid x_{k-1})$, $k=1,\ldots,n$. Consequently, we obtain 
	$$
	\log p(\ul{X}\mid \ul{y})=\log \pi (X_0)+\sum_{k=1}^{n} \log p_{y_{k-1}}(X_{k}\mid X_{k-1}),
	$$
	which implies that the Fisher information matrix has a block-diagonal form:
	$$
	[I(\ul{y})]_{kl}
	=
	\begin{cases} 
	I_k:=
	\E \left[\left(\partial_{y_k}\log p_{y_{k}}(X_{k+1}\mid X_{k})\right)^\top\left(\partial_{y_k}\log p_{y_{k}}(X_{k+1}\mid X_{k})\right)\middle| \ul{Y}=\ul{y}\right] & \text{if}\,\, k=l \\
	0 & \text{if}\,\, k\ne l.
	\end{cases}
	$$ 
	
	Observe that $\mathbf{1}^\top \partial_{\ul{y}}\phi (\ul{y})= \partial_{\ul{y}}\E \left[S_n\mid \ul{Y}=\ul{y}\right]\in \lin (\Y^n,\R)$ hence substituting back the above form of the Fisher information matrix into \eqref{eq:varbound}
	and applying \eqref{eq:CR}
	yields
	\begin{align*}\label{eq:varsum}
	\begin{split}
	\var_{\ul{Y}}(S_n)
	&\ge
	\sum_{k=0}^{n-1} \partial_{y_k}
	\E \left[S_n\mid \ul{Y}\right] 
	I_k^{-1}
	\partial_{y_k}\E \left[S_n\mid \ul{Y}\right]^\top
	\ge \sum_{k=0}^{n-1}\frac{1}{r (I_k)}\lVert 	\partial_{y_k}\E \left[S_n\mid \ul{Y}\right]\rVert^2,
	\end{split}
	\end{align*}
	where $r (I_k)$ refers to the spectral radius of the Fisher operator $I_k$. At least, by \eqref{eq:totvar}, we obtain
	\begin{equation}\label{eq:varsum2}
	\var (S_n) \ge \sum_{k=0}^{n-1} \E \left[\frac{1}{r (I_k)}\lVert\partial_{y_k}\E \left[S_n\mid \ul{Y}\right]\rVert^2\right].
	\end{equation}
	
\end{proof}

Combining the above lemma with Corollary \ref{cor:ekstrom} leads to the following significant conclusion.
\begin{corollary}\label{cor:ekstrom2}
In addition to the conditions of Theorem \ref{thm:MCRE_LLN_and_weakapproaching} and Lemma \ref{lem:VarSn}, assume that the function $\Phi:\X\to\R$ satisfies inequality \eqref{eq:Phi} for some exponent $p>2$. Further, suppose that
$$
\liminf_{n\to\infty}\frac{1}{n}\sum_{k=0}^{n-1}\E \left[\frac{1}{r (I_k)}\lVert\partial_{y_k}\E \left[S_n\mid \ul{Y}\right]\rVert^2\right]>0.
$$
Then it is evident that the sequence $\left(n/\var (S_n)\right)_{n\ge1}$ is bounded. Consequently, by Corollary \ref{cor:ekstrom}, it follows that in addition to the results of Theorem \ref{thm:MCRE_LLN_and_weakapproaching}, 
we also have
$$
\frac{1}{\var (S_n)^{1/2}}S_n\Longrightarrow\mathcal{N} (0,1),\,\, as\,\, n\to\infty
$$
in distribution. Therefore, the Central Limit Theorem holds.
\end{corollary}

To derive a lower bound for $\var(S_n)$, we strongly relied on the Cram\'er-Rao bound, which requires certain differentiability and regularity conditions. However, these differentiability assumptions can be relaxed, either by applying a version of the Cramér–Rao inequality developed for non-differentiable models \cite{cramer_nondiff}, or by using some other information inequality (see, e.g., paragraph 3 in Chapter 30 of \cite{BorovkovCramer}).

It seems that it might be as difficult to evaluate the lower bound in \eqref{eq:varsum2} as to directly verify that
\[
\var(S_n) = O\left(\sum_{k=1}^n \var(\Phi(X_k))\right).
\]
To show that the lower bound on the right-hand side of \eqref{eq:varsum2} is tractable, we establish a more explicit sufficient condition under the assumptions that $\X, \Y \subseteq \R$ and the function $f$ in the iteration \eqref{eq:iter} satisfies some monotonicity condition.
\begin{proposition}\label{prop2}
	Assume that $\X, \Y \subseteq \R$, moreover the function $f:\X \times \Y \times [0,1] \to \R$ defining the iteration \eqref{eq:iter} is continuously differentiable and monotonically increasing in its first two variables. Besides the assumptions of Lemma \ref{lem:VarSn}, assume also that
	\begin{equation}\label{eq:spectral_cond}
		r^* := \sup_{n \in \N} \E \left[r(I_n)\right] = \sup_{n \in \N} \E \left[\left(\partial_y \log p_y (X_{n+1} \mid X_n)\right)^2\right] < \infty,
	\end{equation}
	and there exists a function $g: \Y \times [0,1] \to (0, \infty)$ such that
	\[
	\partial_y f(x, y, u) \ge g(y, u), \quad x, y \in \R,\, u \in [0,1].
	\]
	
	Then, for the variance of the partial sums $S_n = X_1 + \ldots + X_n$, we have
	\[
	\var(S_n) \ge \frac{1}{r^*} \sum_{k=0}^{n-1} \E \left[g(Y_k, \eps_1)\right]^2.
	\]
\end{proposition}
\begin{proof}
	
By the Cauchy-Schwartz inequality, for $k=0,\ldots,n-1$, we have
$$
\E \left[\partial_{y_k}\E \left[S_n\mid \ul{Y}\right]\right]^2
\le \E \left[\frac{1}{r (I_k)}\left(\partial_{y_k}\E \left[S_n\mid \ul{Y}\right]\right)^2\right] \E [r (I_k)]
$$
hence using the inequality \eqref{eq:spectral_cond}, we arrive at
\begin{align}\label{eq:esti}
	\begin{split}
	\var (S_n) \ge \sum_{k=0}^{n-1} \E \left[\frac{1}{r (I_k)}\lVert\partial_{y_k}\E \left[S_n\mid \ul{Y}\right]\rVert^2\right]\ge \frac{1}{r^*}\sum_{k=0}^{n-1}\E \left[\partial_{y_k}\E \left[S_n\mid \ul{Y}\right]\right]^2. 
\end{split}
\end{align}
	
Fix $0 \le k < n$. Then, by the monotonicity assumption, we have $\partial_{y_k} \E \left[X_l \mid \ul{Y} \right] \ge 0$ for any $l = 1, \ldots, n$. Hence, using that the process $(\eps_n)_{n \ge 1}$ is i.i.d. and independent of the environment $(Y_n)_{n \in \N}$, we can estimate as
\begin{align*}
	\partial_{y_k}\E \left[S_n\mid \ul{Y}\right] =
	\sum_{l=1}^n \partial_{y_k}\E \left[X_l\mid \ul{Y}\right]
	\ge \partial_{y_k}\E \left[X_{k+1}\mid \ul{Y}\right]
	= \E \left[\partial_y f(X_{k},Y_k,\eps_{k+1})\mid \ul{Y}\right]
	\ge \E [g(Y_k,\eps_1)\mid \ul{Y}].
\end{align*}
Substituting this estimate back into the inequality \eqref{eq:esti}, we obtain the desired result.

\end{proof}

\begin{remark}
If the environment $(Y_n)_{n \in \Z}$ is stationary, then the lower bound for the variance of $S_n = X_1 + \ldots + X_n$ given in Proposition \ref{prop2} simplifies to
\[
\var(S_n) \ge \frac{n}{r^*} \E \left[g(Y_0, \eps_1)\right]^2,
\]
so the inequality required in the assumptions of Corollary \ref{cor:ekstrom2} is clearly satisfied.
\end{remark}

Using the recent results by Merlev\`ede and Peligrad in \cite{merlevede2020functional}, we can derive the following functional central limit theorem 
for certain functionals of a MCRE.
This result is significantly stronger than the statement formulated in Corollary \ref{cor:ekstrom2}. 
\begin{theorem}\label{thm:MCRE_FCLT}	
	Assume that the conditions of Corollary \ref{cor:ekstrom2} hold, with the exception that instead of the inequality \eqref{eq:mix_decrease} appearing among the conditions of Theorem 2.5, we only require that
	$$
	\sum_{n\ge 1} n^{2/(p-2)}\alpha^Y(n) < \infty.
	$$
	
	Furthermore, for $n \ge 1$, let $v_n(t) = \min \{1 \le k \le n \mid \var(S_k) \ge t\var(S_n)\}$, and for $1 \le k \le n$, define
	$$
	\xi_{k,n} = \frac{\Phi(X_k) - \E[\Phi(X_k)]}{\var(S_n)^{1/2}}.
	$$
	
	Then the sequence of functions $B_n(t) = \sum_{k=1}^{v_n(t)} \xi_{k,n}$, $t \in (0, 1]$, $n \ge 1$, converges in distribution in $D([0, 1])$ (equipped with the uniform topology) to $(B_t)_{t\in [0,1]}$, where $(B_t)_{t\in [0,1]}$ is a standard Brownian motion.
\end{theorem}

\begin{proof}
	
	In the proof, we check the conditions of Corollary 2.2 found on page 3 of \cite{merlevede2020functional}. These conditions can be divided into two groups: conditions on the moments of the random variables $\xi_{k,n}$, $1 \leq k \leq n$, and conditions on the mixing properties of these variables. We first verify the fulfillment of the conditions in the first group.
	
	By Lemma \ref{lem:PhiVbound}, we have $c_1 = \sup_{n\in\N} \E (|\Phi(X_n)|^p) < \infty$, and Lemma \ref{lem:VarSn} ensures that $n/\var(S_n) < c_2$ for some constant $c_2 > 0$ and for all $n \geq 1$. Then, we have
	\begin{align*}
	\lVert \xi_{k,n}\rVert_p &= \frac{1}{\var(S_n)^{1/2}} \lVert \Phi(X_k) - \E[\Phi(X_k)] \rVert_p 
	\leq 
	2 n^{-1/2}\left(\frac{n}{\var(S_n)}\right)^{1/2}  \lVert \Phi(X_k) \rVert_p \leq 2 c_1^{1/r} c_2^{1/2} n^{-1/2},
	\end{align*}
	from which it immediately follows that $\max_{1 \leq k \leq n} \lVert \xi_{k,n} \rVert_p \to 0$ as $n \to \infty$, and furthermore, $\sup_{n \geq 1} \sum_{k=1}^{n} \lVert \xi_{k,n} \rVert_p^2 < \infty$.
	Comparing this with the remark following Corollary 2.2 in \cite{merlevede2020functional}, we conclude that condition (3) and the first part of condition (7) related to moments in \cite{merlevede2020functional} are satisfied. Additionally, condition (1) is inherently satisfied by the array $\{\xi_{k,n} \mid 1 \leq k \leq n,\,\, n \geq 1\}$ due to its definition.
	
	We now proceed to verify the conditions related to the mixing properties. The precise definitions of the weak strong mixing coefficients $\alpha_1(k)$, $\alpha_{1,n}(k)$, and $\alpha_{2,n}(k)$ can also be found in \cite{merlevede2020functional}. For our purposes, it suffices to note that the following sequence of inequalities holds:
	$$
	\alpha_{1,n}(k) \leq \alpha_{2,n}(k) \leq 2\alpha^{\Phi(X)}(k), \quad 1 \leq k \leq n, \,\, n \geq 1,
	$$
	therefore, to verify the second part of condition (7) in \cite{merlevede2020functional} and to satisfy condition (8), it is enough to establish that
	\begin{equation}\label{eq:ezkell}
	\sum_{n \geq 1} n^{2/(p-2)}\alpha^{\Phi(X)}(n) < \infty.	
	\end{equation}
	
	From here, the proof continues by estimating the mixing coefficient $\alpha^{\Phi(X)}$ and follows exactly the same path as described in the proof of Theorem \ref{thm:GeneralResult} and \ref{thm:MCRE_LLN_and_weakapproaching}.
	By Lemma \ref{lem:MCREmixing}, the chain $(X_n)_{n\in\mathbb{N}}$ admits a random iteration representation satisfying the coupling condition with $b(n)=O(e^{-cn^{1/2}})$, $c>0$. Therefore, we can estimate
	\begin{align*}
	\alpha^{\Phi(X)}(n) &\le \alpha^{X}(n) \le \alpha^{Y} (\lfloor n/2\rfloor +1) + b(n-\lfloor n/2\rfloor), 
	\quad n > 2N,
	\end{align*}
	where $N>0$ is as in Lemma \ref{lem:MCREmixing}. Considering that $\sum_{n=1}^{\infty} n^{2/(p-2)} b(n-\lfloor n/2\rfloor) < \infty$, and that the condition $\sum_{n\ge 1} n^{2/(p-2)}\alpha^Y(n) < \infty$ assumed, we obtain that the infinite series appearing in condition \eqref{eq:ezkell} is convergent.
	
	With this, we have verified all of the conditions of Corollary 2.2 in \cite{merlevede2020functional}, from which it follows that the statement formulated in the theorem holds. This completes the proof.
	
\end{proof}

\begin{remark}
In the proof of the above Theorem \ref{thm:MCRE_FCLT}, we do not use Theorem 2.1 of \cite{merlevede2020functional} in its strongest form, but rather we verify only the conditions of Corollary 2.2. It is likely that Theorem \ref{thm:MCRE_FCLT} can be further strengthened and generalized for Markov chains in a random environment satisfying a weak-strong mixing-type condition. However, this is beyond the scope of the present study.
\end{remark}

\section{Single server queuing systems}\label{sec:queuing}

In the early 20th century, Danish engineer Agner Krarup Erlang pioneered what would later be known as queuing theory \cite{erlang1909}. His work at a telephone company, where he developed a mathematical model to determine the minimum number of telephones needed to handle calls efficiently, laid the foundation for this field. Today, queuing theory extends far beyond telecommunications, significantly influencing areas like inventory management, logistics, transportation, industrial engineering, and service design. Notably, it plays a key role in reducing costs within product-service design \cite{Nguyen2017}.

\smallskip
For simplicity, we focus on single-server queuing systems with infinite buffer and first-in, first-out (FIFO) service discipline (see Figure \ref{fig:queue}).  It's worth noting that more complex queuing systems, such as those with multiple servers, can be analyzed using analogous methods.
\begin{figure}[!h]
	\centering
	\begin{tikzpicture} 
		\draw[thick, fill=blue!20] (0,0) -- ++(2cm,0) -- ++(0,-1.5cm) -- ++(-2cm,0);
		\foreach \i in {1,...,4}
		\draw[thick] (2cm-\i*10pt,0) -- +(0,-1.5cm);
		
		\draw[thick, fill=orange!30] (4,-0.75) circle [radius=0.75cm];
		
		\draw[->, thick] (2,-0.75) -- +(36pt,0);
		\draw[<-] (0,-0.75) -- +(-20pt,0) node[left] {Requests};
		
		\node[align=center] at (1cm,-2) {\textbf{Buffer}\\(Infinite capacity)};
		\node[align=center] at (4cm,-2) {\textbf{Server}\\(FIFO discipline)};
		
	\end{tikzpicture}
	\caption{Schematic overview of the single-server queuing system under investigation: $Z_n$ is the time between the arrivals of the $n$-th and $(n+1)$-th customers, $S_n$ is the time to serve the $n$-th customer, and $W_n$ is their waiting time before service.}
	\label{fig:queue}
\end{figure}
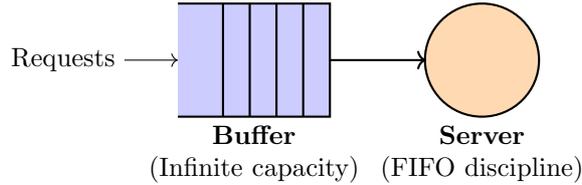
Let the time between the arrival of customers $n+1$ and $n$ is be denoted by $Z_{n+1}$, and the service time for customer $n$ is given by $S_n$, for $n\in\N$. Then, the evolution of the waiting time $W_n$ of customer $n$ can be described by the \emph{Lindley recursion} 
\begin{equation}\label{eq:Lindley}
W_{n+1}=(W_n+S_{n}-Z_{n+1})_+,\,\, n\in\N,
\end{equation}
with $W_0:=0$, meaning that we begin with an empty queue.

\smallskip
The ergodic theory of general state space Markov chains allows to treat the case where $(S_n)_{n\in\N}$, $(Z_n)_{n\in\N}$ are i.i.d.\ sequences, independent of each other. 
However, dependencies frequently arise in queuing networks when the arrival processes intertwine with the departure processes of other queues. Additionally, factors like complex processing operations, including batching or the presence of multiple distinct customer classes, can introduce intricate interdependencies within the system. Consequently, the renewal process assumption for the arrival process, which makes queueing models amenable to simple analysis, no longer holds. 

\smallskip
To the best of our knowledge, Loynes was the first who studied the stability of waiting times
under the assumption that the pair $(S,Z)$ is merely stationary and ergodic \cite{loynes1962}. Stability of $W_n$, $n\in\N$ means here that there exist a unique limit distribution of $W_n$ as $n\to\infty$, whatever the initialization $W_{0}$ is. Loynes introduced the terminology categorizing queues as 'subcritical' when $\E(S_0) < \E(Z_0)$, 'critical' when $\E(S_0) = \E(Z_0)$, and 'supercritical' when $\E(S_0) > \E(Z_0)$. Loynes proved for single-server queuing systems that subcritical queues are stable, supercritical queues are unstable and critical queues can be stable, properly substable, or unstable \cite{loynes1962}. 

\smallskip
Building upon Loynes's foundational work, Gy\"orfi and Morvai expanded and refined the understanding of these queuing systems in \cite{gyorfi2002}. They extended Loynes' result by proving that for subcritical queues, an even stronger version of stability called \emph{forward coupling} holds also true (c.f. Definition \ref{def_fwd_cpl}). Gy\"orfi and Morvai's theorem concerning queues in this general setting reads
as follows:
\begin{theorem}\label{thm:gyorfi}
	Let $\xi_n = S_n-Z_{n+1}$ and assume that the process 
	$(\xi_n)_{n\in\Z}$ is stationary and ergodic with $\E (S_0)<\E (Z_0)$.
	Then $(W_n)_{n\in\N}$ is forward coupled with a stationary and ergodic $(W_n')_{n\in\Z}$ such that  $W_0'=\sup_{n\in\N} Y_n$, where
	$Y_0 = 0$ and $Y_n = \sum_{k=1}^n \xi_{-k}$, $n\ge 1$.
\end{theorem}

While the aforementioned theorems ensure that the distribution of waiting times converges to a well-defined limiting law, they regrettably do not furnish any insights into the properties of this stationary limit, nor do they shed light on the speed of convergence. When addressing the latter question, the only available rate estimate that comes to our aid is encapsulated in an inequality found in Theorem 4 on page 25 of \cite{bborovkov72}, which is expressed as:
\begin{equation}\label{eq:borovkov}
\lVert \law{W_n}-\law{W_0'}\rVert_{TV}\le \P \left(\min_{0<k<n} X_k>\max (W_1,W_0'+\xi_0)\right),
\end{equation}
where $\xi_n$ is as in Theorem \ref{thm:gyorfi}, and $(X_n)_{n\in\N}$ is defined as $X_0 = 0$, $X_n = \sum_{k=1}^{n} \xi_k$, $n\ge 1$. However, it's worth noting that the primary limitation of this formula lies in its practical applicability. Evaluating the probability on the right-hand side of this equation can be a formidable task, rendering it impractical as a concrete and readily usable rate estimate. 

\smallskip
If the inter-arrival times are i.i.d. and the sequences $S$ and $Z$ are independent of each other, then the process $W$ can be viewed as a Markov chain in the random environment $S$, with driving noise $Z$. Conversely, if the service times are i.i.d., and again $S$ and $Z$ are independent, $W$ can be regarded as a Markov chain in the random environment $Z$, with driving noise $S$. Therefore, both of these special cases of queuing systems fall within the theoretical framework outlined in Section \ref{sec:MCRE}.
In \cite{lovas2021ergodic,lovas}, we analyzed such queuing models with an additional G\"artner-Ellis-type condition (see Assumption \ref{as:queue:GE} below or Assumption 4.2 in \cite{lovas}), which is a well-established practice in queuing theory. For instance, in Section 3 of \cite{gyorfi2002}, similar conditions are employed to investigate the exponential tail behavior of the limit distribution of queue length when arrivals exhibit weak dependence. Further justification for the applicability of G\"artner-Ellis-type conditions can be found in Remark 4.3 of \cite{lovas}.

\smallskip
In what follows, we revisit the queuing model studied in \cite{lovas}. Specifically, we consider the case where $S$ and $Z$ are independent, and the latter is an i.i.d. sequence. The reverse case discussed in \cite{lovas2021ergodic} can be treated analogously. We begin by examining the scenario where the sequence of service times is strictly stationary. Regarding this case, we formulate our standing assumptions.
\begin{assumption}\label{as:queue:proc}
	There exists an $M > 0$ such that the sequence of service times $(S_n)_{n\in\Z}$ is a strictly stationary process taking values in $[0,M]$. Furthermore, $(S_n)_{n\in\Z}$ is independent of the sequence $(Z_n)_{n\in\Z}$.
\end{assumption}

\begin{assumption}\label{as:queue:stab}
	The inter-arrival times $(Z_n)_{n\in\Z}$ form an i.i.d. sequence of $\R_{+}$-valued random variables, and $\E[S_0] < \E[Z_1]$ holds.
\end{assumption}

\begin{assumption}\label{as:queue:GE}
There exists $\eta > 0$ such that for all $t \in (-\eta, \eta)$, the limit
		\begin{equation*}
			\Gamma(t) := \lim_{n\to\infty} \frac{1}{n} \log \E e^{t(S_1 + \ldots + S_n)}
		\end{equation*}
		exists and $\Gamma$ is differentiable on $(-\eta, \eta)$.
\end{assumption}

Under Assumptions \ref{as:queue:proc}, \ref{as:queue:stab}, and \ref{as:queue:GE}, by Lemma 4.4 of \cite{lovas}, the sequence of waiting times $(W_n)_{n\in\N}$, defined by the Lindley recursion \eqref{eq:Lindley}, forms a Markov chain in the random environment $(S_n)_{n\in\Z}$ on the state space $\X = \R_{+}$ with parametric kernel
\begin{equation*}
	Q(s, w, A) := \P \left[\left(w + s - Z_1\right)_+ \in A\right], \,
	s \in [0, M], w \in \R_{+}, A \in \B\left(\R_{+}\right).
\end{equation*}
Furthermore, there exists $\bar{t} > 0$ such that, with the following coefficients:
\begin{align*}
	V(w)    &:= e^{\bar{t}w} - 1, \quad w \ge 0,\\	
	\gamma(s) &:= \E\left[e^{\bar{t}(s - Z_1)}\right], \quad s \ge 0,\\
	K       &:= e^{\bar{t}M},
\end{align*}
the drift condition \eqref{eq:Lyapunov} and the long-term contractivity condition \eqref{eq:LT} are satisfied, meaning
$$
[Q(s)V](w) \le \gamma(s)V(w) + K,
$$
and
\begin{equation*}
	\bar{\gamma} := \limsup_{n \to \infty} \left( \E^{1/n} \left[K \prod_{k=1}^{n} \gamma(S_k)\right] \right) < 1.
\end{equation*}

We now introduce an additional assumption on the inter-arrival times, which will be required to establish the minorization condition.
\begin{assumption}\label{as:queue:minor}
	One has $\P \left(Z_0\ge\tau\right)>0$ for $$
	\tau:=M+\frac{4}{\frac{1}{\bar{\ga}^{1/2}}-1}.
	$$
\end{assumption}	
This does not impose a significant restriction, since if $Z_0$ is an unbounded random variable, Assumption \ref{as:queue:minor} is automatically satisfied.

By Lemma 4.6 of \cite{lovas}, under Assumptions \ref{as:queue:proc}, \ref{as:queue:stab}, \ref{as:queue:GE}, and \ref{as:queue:minor}, it can be shown that
there exists $\bar{\beta} \in (0,1)$ such that, for all $s \in [0,M]$, $A \in \B\left(\R_{+}\right)$, and $w \in \stackrel{-1}{V}([0,R(s)])$,
\begin{equation}\label{eq:old_minor}
Q(s, w, A) \ge (1 - \bar{\beta}) \delta_0 (A), \,
\text{where } R(s) = \frac{2 K(s)}{\epsilon \gamma(s)},
\end{equation}
$\epsilon := \left(\frac{1}{\bar{\gamma}^{1/2}} - 1\right)/2$, and $\delta_0$ is the one-point mass concentrated at 0.

\smallskip
The main result of Chapter 4, which discusses queuing theory applications in \cite{lovas}, is Theorem 4.7. It states that under Assumptions \ref{as:queue:proc}, \ref{as:queue:stab}, \ref{as:queue:GE}, and \ref{as:queue:minor}, there exists a probability measure $\mu_{*}$ on $\mathcal{B}(\R_{+})$, independent of the initial length of the queue, such that
\begin{equation}\label{eq:old_res}
	\lVert \law{W_n} - \mu_{*} \rVert_{TV} \leq c_1 e^{-c_2 n^{1/3}},
\end{equation}
for some $c_1, c_2 > 0$. Furthermore, if $\left(S_n\right)_{n\in\Z}$ is ergodic, then for an arbitrary measurable and bounded function $\Phi: \R_{+} \to \R$,
\begin{equation}\label{torta2}
	\frac{\Phi(W_0) + \ldots + \Phi(W_{n-1})}{n} \to \int_{\R_{+}} \Phi(z) \mu_{*}(\dint z),
\end{equation}
in $L^p$, for all $1 \le p < \infty$. 

Using Theorem \ref{thm:stac_forward_coupling} and Corollary \ref{cor_fwd_dtv}, we can prove even more. In the queuing model we study, for instance, the sequence of waiting times $(W_n)_{n\in\N}$ is forward coupled with a strictly stationary sequence $(W_n^*)_{n\in\N}$. This does not require assuming the ergodicity of the service time sequence, and we obtain a tractable upper bound on the tail probability of the random coupling time. Finally, we demonstrate a faster convergence rate than that presented in the estimate \eqref{eq:old_res}. The following theorem addresses this result.
\begin{theorem}\label{thm:queue:stac}
	Let Assumptions \ref{as:queue:proc}, \ref{as:queue:stab}, \ref{as:queue:GE}, and \ref{as:queue:minor} be in force. Then there exists a stationary process $(S_n, W_n^*)_{n\in\Z}$ that satisfies the Lindley recursion \eqref{eq:Lindley}. Moreover, appropriate versions of the processes $(W_n)_{n\in\N}$ and $(W_n^*)_{n\in\N}$ are forward coupled. For the tail probability of the random coupling time $\tau$, we have the bound
	$$
	\P(\tau > n) \leq c_1 e^{-c_2 n^{1/2}},
	$$
	for suitable constants $c_1, c_2 > 0$, depending only on the quantities $\bar{\beta}$, $\bar{\gamma}$, and $\epsilon$. Furthermore, for the total variation distance of $\law{W_n}$ and $\law{W_n^*}$ we have the following estimate:
	$$
	\lVert \law{W_n} - \law{W_n^\ast} \rVert_{TV}
	\leq
	2 c_1 e^{-c_2 n^{1/2}},
	$$
	with the same constants $c_1$ and $c_2$.
\end{theorem}
\begin{proof}
	The sequence of service times $(S_n)_{n\in\N}$ is strictly stationary, so Assumption \ref{as:dynamics} A) simplifies to the \eqref{eq:LT} long-term contractivity condition used in our earlier paper \cite{lovas}. We have already shown above that this holds under the conditions of the theorem. From the minorization condition \eqref{eq:old_minor}, part B) of Assumption \ref{as:dynamics} follows. Thus, we can apply Theorem \ref{thm:stac_forward_coupling}, which states that there exists a stationary process $(S_n, W_n^*)_{n\in\Z}$ that satisfies the Lindley recursion \eqref{eq:Lindley}. Moreover, appropriate versions of the processes $(W_n)_{n\in\N}$ and $(W_n^*)_{n\in\N}$ are forward coupled. Considering the specific form of the Lyapunov function in the drift condition, $V(w) = e^{\bar{t}w} - 1$, and the initial condition $W_0 = 0$, we obtain
	$$
	\E (V(W_0)) = 0.
	$$
	Therefore, for the tail probability of the random coupling time $\tau$, we have
	$$
	\P(\tau > n) \leq c_1 (1 + \E (V(W_0))) e^{-c_2 n^{1/2}} = c_1 e^{-c_2 n^{1/2}},
	$$
	as stated. Finally, the estimate for the total variation distance between $\law{W_n}$ and $\law{W_n^*}$ immediately follows from Corollary \ref{cor_fwd_dtv}.
\end{proof}

In the remaining part of this section, we relax the assumption that the sequence of service times $(S_n)_{n\in\Z}$ is stationary. We only assume that it is a sequence of weakly dependent variables with sufficiently favorable mixing properties. Accordingly, our assumptions will be as follows.

\begin{assumption}\label{as:queue:Z} 
	We assume that the inter-arrival time sequence $(Z_n)_{n \geq 1}$ is i.i.d. and independent of the service time sequence $(S_n)_{n \in \mathbb{N}}$.
	Moreover, we also assume that $\mathbb{E}(Z_0^2) < \infty$.
\end{assumption}

\begin{assumption}\label{as:queue:S} 
	We assume that the service time sequence $(S_n)_{n \in \mathbb{N}}$ takes values in $\mathcal{Y} = [0, M]$. 
	For $t\ge 0$, we define the function
	\[
	\Lambda(t) = \limsup_{n \to \infty} \sup_{j \in \mathbb{N}} 
	\frac{1}{n} \log \mathbb{E} \left[ \exp \left( t \sum_{k=0}^n (S_{k+j} - Z_{k+j+1}) \right) \right].
	\]
	Assume that there exists a parameter $\bar{t} > 0$ such that $\Lambda(\bar{t}) < 0$.
\end{assumption}

\begin{remark}\label{rem:munder}
	In Lemma 4.4 of \cite{lovas}, under Assumption \ref{as:queue:GE}, we showed that there exists a parameter $\bar{t} > 0$ such that 
	in Assumption \ref{as:queue:S}, $\Lambda(\bar{t}) < 0$. Although such a condition is quite standard in both queuing theory and large deviation theory, verifying the differentiability of the function $t \mapsto \Gamma(t)$ in Assumption \ref{as:queue:GE} can be challenging in practice. 
	Moreover, the proof of Lemma 4.4 in \cite{lovas} does not generalize to the case of a non-stationary sequence $(S_n)_{n \in \mathbb{N}}$, since the definition of the function $\Lambda$ includes a supremum, which prevents the application of convexity-based arguments.
	
	For fixed $n, j \in \mathbb{N}$, define
	\[
	\lambda_{n,j}(t) = \frac{1}{n} \log \mathbb{E} \left[ \exp \left( t \sum_{k=0}^n (S_{k+j} - Z_{k+j+1}) \right) \right].
	\]
	If $\limsup_{n \to \infty} \mathbb{E}(S_n) < \mathbb{E}(Z_1)$, it can be shown that for sufficiently large $n$, $\lambda_{n,j}'(0) < 0$. Additionally, if the mixing coefficients $\alpha^S(n)$ decay sufficiently fast to zero, it can be established that $\lambda_{n,j}''(0) < \infty$. However, this alone does not ensure the existence of a parameter $\bar{t} > 0$ such that $\Lambda(\bar{t}) < 0$. 
\end{remark}

The goal of the next proposition is to provide a sufficient condition under which the otherwise difficult-to-verify requirement on $\Lambda(t)$ in Assumption~\ref{as:queue:S} holds. This condition relies on a mixing property stronger than $\alpha$-mixing, known as $\psi$-mixing. In analogy with the dependence measure $\alpha(\mathcal{G},\mathcal{H})$ introduced in Section~\ref{sec:mixtrans}, we define
\begin{equation}\label{eq:dep_psi}
	\psi (\G,\HH) = \sup \left\{
	\left|\frac{\P (G\cap H)}{\P (G) \P (H)} -1\right|\middle|
	\P (G)>0,\,\P(\HH)>0,\,G\in\G,\,H\in\HH
	\right\}.
\end{equation}
Furthermore, for an arbitrary sequence of random variables $(W_t)_{t\in\Z}$, 
the sequence of $\psi$-mixing coefficients is defined as
$$
\psi^W (n)=\sup_{j\in\Z}\psi (\F_{-\infty,j}^W,\F_{j+n,\infty}^W),
\quad n\in\N.
$$ 
We say that the process $(W_n)_{n\in\Z}$ is $\psi$-mixing if $\psi^W(n)\to 0$ as $n\to\infty$.

It is straightforward to show that every $\psi$-mixing process is also $\alpha$-mixing, but the converse does not hold. Moreover, $\psi$-mixing is one of the strongest forms of strong mixing conditions, implying several others such as $\alpha$-, $\beta$-, $\phi$-, and $\rho$-mixing. Nevertheless, the class of $\psi$-mixing processes remains of considerable interest. For instance, in the case of strictly stationary, finite-state Markov chains, the notions of $\alpha$- and $\psi$-mixing coincide (See Theorem 3.1 in \cite{Bradley2005}).

To formulate the next proposition, we need the notion of stochastic ordering. A real-valued random variable $W_1$ is said to be stochastically smaller than another random variable $W_2$, denoted by $W_1 \preceq W_2$, if
$$
\P (W_1>x)\le \P (W_2>x),\quad x\in\R,
$$
or, equivalently,
$$
\E [u (W_1)]\le \E[u (W_2)],
$$
for every monotonically increasing function $u: \R\to\R$.
\begin{proposition}\label{prop:psi_mix}
	Let the sequence of inter-arrival times $(Z_n)_{n \ge 1}$ be i.i.d. and independent of the service time sequence $(S_n)_{n \in \mathbb{N}}$, which is assumed to be $\psi$-mixing and takes values in $\mathcal{Y} = [0, M]$. Suppose furthermore that there exists a random variable $S^*$ such that $\E [S^*]<\E [Z_1]$, moreover
	$S_n \preceq S^*$ for all $n \in \N$.
	
	Then there exists $\bar{t} > 0$ such that 
	$$
	\Lambda(\bar{t}) < 0,
	$$ 
	where $\Lambda(t)$ is the rate function defined in Assumption~\ref{as:queue:S}.
\end{proposition}
\begin{proof}
	Let $j\in\N$ and $n\ge 1$ be arbitrary but fixed, and define
	$W_k = e^{t (S_{k-1}-Z_{k})}$, $k=1,\ldots,n$. Our goal is to estimate $\E [W_{j+1}\cdots W_{j+n}]$. For integers $p,q$ satisfying $pq\le n < p (q+1)$, which allows us to write
	$$
	W_{j+1}\cdots W_{j+n} =\eta_1\cdots \eta_p \Delta_n,
	$$
	where $\eta_k= \prod_{l=0}^{q-1} W_{j+k+lp}$, $k=1,\ldots, p$, and $\Delta_n=W_{j+pq+1}\cdots W_{j+n}$.
	
	Using $\Delta_n\le e^{tM (p-1)}$ and H\"older-inequality, we have
	\begin{equation}\label{eq:est16}
	\E [W_{j+1}\cdots W_{j+n}]\le e^{tM (p-1)}\prod_{k=1}^{p}\E^{1/p}[\eta_k^p].
	\end{equation}
	For any fixed $k$, by independence, we have 
	$$
	\E [\eta_k^p]=\E\left[\prod_{l=0}^{q-1}e^{tp (S_{j+k-1+lp}-Z_{j+k+lp})}\right]=\E\left[e^{-tpZ_1}\right]^q\E\left[\prod_{l=0}^{q-1}e^{tp S_{j+k-1+lp}}\right].
	$$
	Using layer cake representation, and that $S_n\preceq S^*$, for $n\in\N$, we can write
	\begin{align*}
		\E\left[\prod_{l=0}^{q-1}e^{tp S_{j+k-1+lp}}\right]
		&=
		\int_{[0,\infty)^q}\P (e^{tp S_{j+k-1+lp}}\ge \tau_l,\,l=0,\ldots,q-1)
		\,\dint \tau_0\ldots\dint \tau_{q-1}
		\\
		&\le
		 (1+\psi^S (p))^{q-1} \prod_{l=0}^{q-1}
		 \int_{[0,\infty)}\P (e^{tp S_{j+k-1+lp}}\ge \tau_l)
		 \,\dint \tau_l
		 \\
		&=
		(1+\psi^S (p))^{q-1} \prod_{l=0}^{q-1} \E\left[e^{tp S_{j+k-1+lp}}\right]\le (1+\psi^S (p))^{q-1}\E\left[e^{tp S^*}\right]^q.
	\end{align*}
	To sum up, and taking the supremum in $j$, we obtain the following estimate:
	\begin{equation}\label{eq:est17}
		\sup_{j\in\N}\E [W_{j+1}\cdots W_{j+n}]\le e^{tM (p-1)} (1+\psi^S (p))^{q-1}\E [e^{tp (S^*-Z_1)}]^q.
	\end{equation}
	
	Since $\left.\frac{\dint}{\dint h}\E\left[e^{h(S^*-Z_1)}\right]\right|_{t=0}= \E [S^*-Z_1]<0$, for arbitrary $m$ satisfying $\E [S^*-Z_1]<m<0$ there exists
	$h_0>0$ such that
	$$
	\E \left[e^{h(S^*-Z_1)}\right]<1+mh,\quad 0<h<h_0.
	$$
	
	Let $h_0'\in (0,h_0)$ be fixed. By the $\psi$-mixing property of the sequence $(S_n)_{n\in\N}$, there exists $p_0\ge 1$ such that 
	$$
	\zeta:=(1+\psi^S (p_0))(1+m h_0')<1.
	$$
	Using \eqref{eq:est17} with $t=\bar{t}:=h_0'/p_0$ and $q_n=\lfrf{\frac{n}{p_0}}$, $n\ge 1$, we obtain
	$$
	\sup_{j\in\N}\E [W_{j+1}\cdots W_{j+n}]\le e^{h_0' M}\zeta^{q_n}\le  e^{h_0' M}\zeta^{-1}\zeta^{n/p_0},
	$$
	and thus
	$$
	\sup_{j \in \mathbb{N}} 
	\frac{1}{n} \log \mathbb{E} \left[ \exp \left( \bar{t} \sum_{k=0}^n (S_{k+j} - Z_{k+j+1}) \right) \right]
	\le \frac{n+1}{n p_0}\log \zeta + \frac{1}{n}\log (e^{h_0' M}\zeta^{-1})
	$$
	which immediately implies that
	$$
	\Lambda (\bar{t})=\limsup_{n\to\infty}\sup_{j \in \mathbb{N}} 
	\frac{1}{n} \log \mathbb{E} \left[ \exp \left( \bar{t} \sum_{k=0}^n (S_{k+j} - Z_{k+j+1}) \right) \right]\le \frac{1}{p_0}\log \zeta<0.
	$$
\end{proof}

It is conceivable that the requirement on $\Lambda(t)$ in Assumption~\ref{as:queue:S} could be verified under mixing conditions on the sequence $(S_n)_{n\in\N}$ that are weaker than $\psi$-mixing. In particular, \cite{Bryc1996LDP} addresses large deviation theorems under strong mixing. The authors showed that for stationary processes with a hyper-exponential mixing rate—specifically, $\alpha(n) \ll \exp(-n (\log n)^{1+\delta})$ for some $\delta > 0$—the existence of a limit analogous to the one in the definition of the rate function $\Lambda$ is ensured (see Theorem 1 in \cite{Bryc1996LDP}). However, they also provided counterexamples based on Doeblin recurrent, irreducible Markov chains with countable state spaces to demonstrate that this decay condition on the mixing coefficient cannot be significantly weakened.

\begin{lemma}\label{lem:queue:driftLT}
	Under Assumption \ref{as:queue:Z}, the sequence of waiting times $(W_n)_{n\in\N}$, defined by the Lindley recursion \eqref{eq:Lindley}, forms a Markov chain in the random environment $(S_n)_{n\in\Z}$ on the state space $\X = \R_{+}$ with the parametric kernel
	\begin{equation*}
		Q(s, w, A) := \P \left[\left(w + s - Z_1\right)_+ \in A\right], \,\,
		s , w \in \R_{+},\,\, A \in \B\left(\R_{+}\right).
	\end{equation*}
	For the Lyapunov function $V(w) = e^{t w}-1$, $w\in\R_{+}$, the drift condition holds for any choice of $t>0$:
	$$
	[Q(s)V](w)\le \gamma (s) V(w)+K(s)
	$$
	with $\gamma (s) = K(s) = \E\left[e^{t(s-Z_1)}\right]$, $s \ge 0$.
\end{lemma}

\begin{proof}

The sequence of waiting times $(W_n)_{n\in\N}$, governed by the Lindley recursion \eqref{eq:Lindley}, constitutes a non-linear autoregressive process of the form \eqref{eq:iter}. In this framework, $(S_n)_{n\in\N}$ represents the sequence of exogenous covariates, while $(Z_n)_{n\in\N}$ plays the role of the noise process $(\eps_n)_{n\in\N}$.
Furthermore, by  Assumption \ref{as:queue:Z}, $(Z_n)_{n\in\N}$ is i.i.d. and independent of $(S_n)_{n\in\N}$, and thus the identity in the form \eqref{eq:iter_con_ex} holds.
Consequently, $(W_n)_{n\in\N}$ can be interpreted as a Markov chain in a random environment defined by $(S_n)_{n\in\N}$, with the corresponding parametric kernel given by
$$
Q(s,w,A) = \P((s+w-Z_1) \in A)_+, \quad s,w \in \R_{+},\, A \in \B(\R_{+}).
$$

Let $t > 0$ be arbitrary and define $V(w) = e^{tw} - 1$.  
Using the inequality $e^{(x)_+} - 1 \leq e^x$ for all $x \in \mathbb{R}$, we obtain the drift condition:
\begin{align*}
	[Q(s)V](w) &= \mathbb{E}\left[V((w+s-Z_1)_+)\right] = \mathbb{E}\left[e^{t(w+s-Z_1)_+}-1\right] \\
	&\leq \mathbb{E}\left[e^{t(w+s-Z_1)}\right] = \mathbb{E}\left[e^{t(s-Z_1)}\right] \left(e^{tw}-1\right) + \mathbb{E}\left[e^{t(s-Z_1)}\right]
	\\
	&=\gamma (s)V(w)+K(s),
\end{align*}
where $\gamma(s) = K(s) = \mathbb{E}\left[e^{t(s-Z_1)}\right]$, $s \geq 0$ which completes the proof.
\end{proof}

The following lemma, which ensures the fulfillment of part B) of Assumption \ref{as:dynamics}, essentially coincides with Lemma 4.6 in \cite{lovas}. 
We include the proof here solely for the sake of completeness. 
\begin{lemma}\label{lem:q:tech}
	Assume that $\P (Z_1\ge M+\tau)>0$ holds for some $\tau>0$. Then, there exists $\bar{\beta} \in (0,1)$ such that, for all $s \in [0,M]$ and $w \in [0,\tau]$,
	$$
	Q(s,w,A) \geq (1 - \bar{\beta}) \delta_0(A), \,\, A \in \mathcal{B}(\mathbb{R}_{+}),
	$$
	where $\delta_0$ denotes the Dirac measure concentrated at $0$.
\end{lemma}

\begin{proof}
	Let $s \in [0,M]$, $w \in [0,\tau]$, and $A\in\B(\R_{+})$. We can write
	\begin{align*}
		Q(s,w,A) &= \P \left(\left[w + s - Z_1\right]_+ \in A \right) \geq \P \left(\left[w + s - Z_1\right]_+ = 0 \right)\delta_0(A) \\
		&= \left(1 - \P \left(s + w - Z_1 > 0 \right)\right)\delta_0(A) \\
		&\geq \left(1 - \P \left(M + \tau - Z_1 > 0 \right)\right)\delta_0(A),
	\end{align*}
		which shows that the desired inequality holds for any choice of $\bar{\beta}$ satisfying 
		$$
		\P\left(Z_1 < M + \tau \right) < \bar{\beta} < 1.
		$$
\end{proof}

Let Assumptions \ref{as:queue:Z} and \ref{as:queue:S} be in force, and let $\bar{t} > 0$ be as specified in Assumption \ref{as:queue:S}. Then, by Lemma \ref{lem:queue:driftLT}, for the Lyapunov function $V(w) = e^{\bar{t}w} - 1$, $w \ge 0$, the drift condition \eqref{eq:Lyapunov} holds with $\gamma(s) = K(s) = \E [e^{\bar{t}(s - Z_1)}]$, for $s \in \Y = [0, M]$. Furthermore, since $\gamma(s) = K(s) \le e^{\bar{t}M}$, we have
\[
\E \left[ K(S_j) \prod_{k=1}^n \gamma(S_{k+j}) \right] \le e^{\bar{t}M(n+1)},
\]
for all $j \in \N$ and $n \ge 1$, hence the integrability condition in Assumption \ref{as:dynamics} is satisfied. Moreover, by Assumptions \ref{as:queue:Z} and \ref{as:queue:S},
\[
\bar{\gamma} = \limsup_{n \to \infty} \sup_{j \ge -1} \E^{1/n} \left[
\exp\left( \bar{t} \sum_{k=0}^n (S_{k+j} - Z_{k+j+1}) \right)
\right] = e^{\Lambda(\bar{t})} < 1,
\]
which implies that part~A) of Assumption \ref{as:dynamics} also holds.

Suppose that for some $0 < r < 1/\bar{\gamma}$ and $\tau := \frac{1}{\bar{t}} \log\left(1 + \frac{2}{r}\right)$,
\[
\P\left(Z_1 \ge M + \tau\right) > 0.
\]
We now verify part~B) of Assumption~\ref{as:dynamics}. We have $R(s) = \frac{2K(s)}{r\gamma(s)} = 2/r$ since $\gamma(s) = K(s)$ for all $s \in [0, M]$. By Lemma~\ref{lem:q:tech}, there exists $\bar{\beta} \in (0, 1)$ such that for all $s \in [0, M]$ and $w \in [0, \tau]$,
\[
Q(s, w, A) \ge (1 - \bar{\beta}) \delta_0(A), \quad A \in \mathcal{B}(\mathbb{R}_+),
\]
where $\delta_0$ denotes the Dirac measure concentrated at $0$. Note that $w \in [0, \tau]$ if and only if $V(w) \in [0, 2/r] = [0, R(s)]$, for $s \in [0, M]$. Therefore, the minorization condition~\eqref{eq:smallset} holds with $\kappa_R(s, \cdot) = \delta_0(\cdot)$ for $s \in [0, M]$. Furthermore, the minorization coefficient $\beta(R(s), s) = \bar{\beta} \in (0, 1)$ is independent of $s$, hence part~B) of Assumption~\ref{as:dynamics} is also satisfied.

\begin{theorem}\label{thm:queue:main}
	Assume that the sequences $(S_n)_{n \in \N}$ and $(Z_n)_{n \ge 1}$ satisfy Assumptions \ref{as:queue:Z} and \ref{as:queue:S}, moreover for some $0 < r < 1/\bar{\gamma}$,
	$$
	\P \left(Z_1 \ge M + \frac{1}{\bar{t}}\log \left(1+\frac{2}{r}\right)\right) > 0,
	$$
	where $\bar{t}$ is as in Assumption \ref{as:queue:S}, and $\bar{\gamma} = \exp (\Lambda (\bar{t}))$. 

	Then for the sequence of waiting times $(W_n)_{n\in\N}$, defined by the Lindley recursion \eqref{eq:Lindley}, the following statements hold.
	\begin{enumerate}[A)]
		\item 	If the sequence of service times $(S_n)_{n \in \N}$ is strongly mixing, then the $L^1$ weak law of large numbers holds:
		\[
		\frac{1}{n} \sum_{k=1}^{n} \left[ W_k - \E (W_k) \right] \stackrel{L^1}{\to} 0, \quad \text{as} \quad n \to \infty.
		\]
		
		\item If there exists $c>0$ and $\kappa>1$ such that $\alpha^S(n)\le c n^{-\kappa}$, $n\ge 1$, then
		\[
		\frac{1}{n} \sum_{k=1}^{n} \left[ W_k - \E (W_k) \right] \stackrel{\Pas}{\to} 0, \quad \text{as} \quad n \to \infty.
		\]
		
		\item If the sequence of mixing coefficients $(\alpha^S(n))_{n \in \N}$ satisfies the condition 
		\[
		\sum_{n \in \N} (n+1)^2 \alpha^S(n)^\delta < \infty
		\]
		for some exponent $\delta \in (0,1)$, then there exists a constant $\sigma > 0$ such that
		\[
		\limsup_{n \to \infty} \P \left( \frac{1}{\sqrt{n}} \left| \sum_{k=1}^{n} \left( W_k - \mathbb{E}(W_k) \right) \right| \ge a \right)
		\leq \int_{\mathbb{R}} \ind_{[-a,a]^c}(\sigma t) \frac{1}{\sqrt{2\pi}} e^{-\frac{t^2}{2}} \, \mathrm{d}t, \quad a > 0.
		\]
		
	\end{enumerate}	
\end{theorem}

\begin{proof}
	From Lemma \ref{lem:queue:driftLT}, it follows that the sequence of waiting times $(W_n)_{n\in\N}$ is a Markov chain in a random environment $(S_n)_{n \in \N}$, which satisfies the drift condition \eqref{eq:Lyapunov} with $V(w) = e^{tw} - 1$ and $\gamma(s) = K(s) = \E \left[e^{t(s-Z_0)}\right]$. 
	
	Given that $S_n \in [0, M]$ for all $n \in \N$, it holds that 
	$
	\E \left[ K(S_j) \prod_{k=1}^{n} \gamma(S_{k+j}) \right] \leq e^{tM(n+1)},
	$
	which implies that the integrability condition in Assumption \ref{as:dynamics} is trivially satisfied, i.e., 
	\[
	\E \left[ K(S_j) \prod_{k=1}^{n} \gamma(S_{k+j}) \right] < \infty \quad \text{for all} \quad j \in \N \text{ and } n \geq 1.
	\]
	
	Moreover, by Assumption \ref{as:queue:S}, the parameter $t > 0$ can be chosen such that 
	\[
	\bar{\gamma} = \limsup_{n \to \infty} \sup_{j \in \N} \E^{1/n} \left[ \prod_{k=0}^{n} \gamma(S_{k+j}) \right] < 1,
	\]
	which ensures that part A) of Assumption \ref{as:dynamics} is satisfied. 
	
	Notice that, since $\gamma = K$, hence by Lemma \ref{lem:q:tech}, the parametric kernel $Q$ satisfies the minorization condition \eqref{eq:smallset} with
	$$
	R(s) = \frac{2K(s)}{r\gamma(s)} = \frac{2}{r}, \,\, \beta(R(s),s) = \bar{\beta} < 1, \,\, \text{and} \,\, \kappa_{R(s)}(s,A) = \delta_0(A).
	$$
	We can conclude that part B) of Assumption \ref{as:dynamics} also holds.
	
	In Theorem \ref{thm:MCRE_LLN_and_weakapproaching}, the inequality \eqref{eq:Phi} is satisfied with $\Phi = \mathrm{id}_{[0,\infty)}$ for any $p > 1$, while the deterministic initial condition $W_0 = 0$ obviously ensures $\E V(W_0) < \infty$.
	
	We have thus shown that all conditions of Theorem \ref{thm:MCRE_LLN_and_weakapproaching} are met, from which the assertions of the present theorem follow.
\end{proof}

\smallskip 
In what follows, we will demonstrate that the seemingly complex and technical conditions of Lemma \ref{lem:VarSn} and Theorem \ref{thm:MCRE_FCLT} boil down to easily verifiable and natural conditions concerning the sequence $(S_n)_{n \in \mathbb{N}}$ and the inter-arrival time distribution $Z_1$. Utilizing this framework, we establish the functional central limit theorem for the sequence of waiting times $(W_n)_{n \in \mathbb{N}}$.

\smallskip 
Assume that the law of $Z_1$ is absolutely continuous with respect to the Lebesgue measure, i.e., $\law{Z_1}(\dint z)=f_{Z_1}(z)\,\dint z$, where $f_{Z_1}(z) = 0$ for $z \leq 0$. For any Borel set $A \in \mathcal{B}(\mathbb{R}_+)$, we can express $Q(s, w, A)$ as follows:
\begin{align*}
	Q(s,w,A) &= \int_{[0,\infty)} \ind_A\big((w+s-z)_+\big) f_{Z_1}(z)\,\dint z \\
	&= \ind_A(0) \, \mathbb{P}(Z_1 > w + s) + \int_0^{w+s} \ind_A(w+s-z) f_{Z_1}(z)\, \dint z \\
	&= \ind_A(0) \, \mathbb{P}(Z_1 > w + s) + \int_{[0,\infty)} \ind_A(z) f_{Z_1}(w + s - z)\, \dint z \\
	&= \int_{\mathcal{X} = \mathbb{R}_+} \ind_A(z) \left( \mathbb{P}(Z_1 > w + s) \ind_{\{0\}}(z) + f_{Z_1}(w + s - z) \ind_{(0,\infty)}(z) \right) \dint \nu(z),
\end{align*}
where $\nu(\dint z) = \delta_0(\dint z) + \dint z$.

To sum up, we obtained that there exists a Borel measure $\nu$ on $\mathcal{B}(\mathcal{X})$, such that $Q(s, w, \cdot) \ll \nu$ for all $(s, w) \in \mathcal{Y} \times \mathcal{X}$, where $\mathcal{X} = \mathbb{R}_+$ and $\mathcal{Y} = [0, M]$. Moreover, the transition densities are given by
\begin{equation}\label{eq:trans}
	p_s(z \mid w) = \frac{\dint Q(s,w,\cdot)}{\dint \nu}(z) = \mathbb{P}(Z_1 > w + s) \ind_{\{0\}}(z) + f_{Z_1}(w + s - z) \ind_{(0,\infty)}(z).
\end{equation}

For a fixed \( n \geq 1 \), let \( \ul{w} = [w_1, \ldots, w_n] \) and \( \ul{s} = [s_0, \ldots, s_{n-1}] \), and define the likelihood function as
\begin{equation}\label{eq:q:L}
	p(\ul{w} \mid \ul{s}) = \prod_{k=1}^{n} p_{s_{k-1}}(w_k \mid w_{k-1}).
\end{equation}

If the density function \( f_{Z_1}: \mathbb{R} \to [0, \infty) \) is continuously differentiable everywhere, then at every point \( \ul{w} \in [0, \infty)^n \) where \( p(\ul{w} \mid \ul{s}) > 0 \), the derivative \( \partial_{s_i} \log p(\ul{w} \mid \ul{s}) = \partial_{s_i} p_{s_i}(w_{i+1} \mid w_i) \) for \( i = 0, \ldots, n-1 \) exists and is finite. Furthermore, we have
\begin{equation}\label{eq:Q:CR:interchange}
	\frac{\partial}{\partial s_i} \int_{\mathcal{X}^n} (w_1 + \ldots + w_n) p(\ul{w} \mid \ul{s}) \, \dint \ul{w}
	=
	\int_{\mathcal{X}^n} (w_1 + \ldots + w_n) \, \partial_{s_i} p(\ul{w} \mid \ul{s}) \, \dint \ul{w},
\end{equation}
hence the regularity conditions required for the Cram\'er-Rao inequality are satisfied.

\begin{lemma}\label{lem:q:fclt}
	Assume that $f_{Z_1}: [0, \infty) \to [0, \infty)$ is continuously differentiable everywhere, and 
	\[
	\int_{[0, \infty)}
	\frac{f_{Z_1}'(z)^2}{f_{Z_1}(z)}
	\,\mathrm{d}z < \infty.
	\]
	
	Then there exists a constant $r^*>0$ such that for any $n \geq 1$, and $\Sigma_n = \sum_{k=1}^{n} W_k$, the following inequality holds:
	\[
	\var(\Sigma_n) \geq \frac{1}{r^*}\sum_{k=0}^{n-1}\P (S_k>Z_1)^2.
	\]
\end{lemma}

\begin{proof}
	Fix $n \geq 1$ and let $\ul{s} = [s_0, \ldots, s_{n-1}] \in [0, M]^n$. Since $(S_n)_{n \in \mathbb{N}}$ is a scalar-valued process, the Fisher information matrix is diagonal. Therefore, using \eqref{eq:trans}, for $k = 0, \ldots, n-1$, we have
	\begin{align*}
		r(I(s_k)) &= [I(\ul{s})]_{kk} = \mathbb{E}\left[ \left( \partial_{s_k} \log p_{s_k}(W_{k+1} \mid W_k) \right)^2 \middle| \ul{S} = \ul{s} \right] 
		\\
		&= \mathbb{E} \left[ \int_{\mathbb{R}} \frac{f_{Z_1}(W_k + s_k)^2}{\mathbb{P}(Z_1 > W_k + s_k)} \ind_{\{0\}}(z) + \frac{f_{Z_1}'(W_k + s_k - z)^2}{f_{Z_1}(W_k + s_k - z)} \ind_{(0, \infty)}(z) \, (\delta_0 (\mathrm{d}z) + \mathrm{d}z) \middle| \ul{S} = \ul{s} \right]
		\\
		&= \mathbb{E} \left[ \frac{f_{Z_1}(W_k + s_k)^2}{\mathbb{P}(Z_1 > W_k + s_k)} + \int_{[0, \infty)} \frac{f_{Z_1}'(W_k + s_k - z)^2}{f_{Z_1}(W_k + s_k - z)} \, \mathrm{d}z \middle| \ul{S} = \ul{s} \right]
		\\
		&\leq \sup_{z \in [0, \infty)} \frac{f_{Z_1}(z)^2}{\mathbb{P}(Z_1 > z)} + \int_{[0, \infty)} \frac{f_{Z_1}'(z)^2}{f_{Z_1}(z)} \, \mathrm{d}z.
	\end{align*}
	
	By l'Hôpital's rule, we have $\lim_{z \to \infty} \frac{f_{Z_1}(z)^2}{\mathbb{P}(Z_1 > z)} = \lim_{z \to \infty} -f_{Z_1}'(z) = 0$, which implies that 
	\[
	\sup_{z \in [0, \infty)} \frac{f_{Z_1}(z)^2}{\mathbb{P}(Z_1 > z)} < \infty.
	\]
	Therefore, there exists an upper bound for $r(I(s_k))$ that is independent of both $k$ and $n$; let this bound be denoted by $r^*$.
	
	Since the mapping $(s, w, z) \mapsto (s + w - z)_+$, defining the Lindley recursion \eqref{eq:Lindley}, is monotonic in both $s$ and $w$, we observe that for each $1 \leq j \leq n$ and $0 \leq k \leq n-1$, the derivative $\partial_{s_k} \mathbb{E} \left[ W_j \mid \ul{S} \right] \geq 0$. Therefore,
	\begin{align*}
		\lVert \partial_{s_k} \mathbb{E} \left[ \Sigma_n \mid \ul{S} \right] \rVert^2 &= \left( \sum_{j=1}^{n} \partial_{s_k} \mathbb{E} \left[ W_j \mid \ul{S} \right] \right)^2
		\geq \left( \partial_{s_k} \mathbb{E} \left[ W_{k+1} \mid \ul{S} \right] \right)^2 = \mathbb{P}(W_k + S_k - Z_{k+1} > 0 \mid \ul{S})^2
		\\
		&\geq \mathbb{P}(S_k - Z_1 > 0 \mid \ul{S})^2.
	\end{align*}
	
	By Lemma \ref{lem:VarSn} and the Cauchy-Schwartz inequality, using the bounds obtained above, we get
	\begin{align*}
		\var(\Sigma_n) &\geq \sum_{k=0}^{n-1} \mathbb{E} \left[ \frac{1}{r(I_k)} \lVert \partial_{s_k} \mathbb{E} \left[ \Sigma_n \mid \ul{S} \right] \rVert^2 \right]
		\\
		&\geq \sum_{k=0}^{n-1} \frac{1}{r^*} \mathbb{E} \left[ \mathbb{P}(S_k - Z_1 > 0 \mid \ul{S})^2 \right]
		\geq \frac{1}{r^*}\sum_{k=0}^{n-1}\P (S_k>Z_1)^2,
	\end{align*}
	which completes the proof.
\end{proof}

\begin{theorem}\label{thm:queue:FCLT}
	Beyond the conditions established in Theorem \ref{thm:queue:main}, suppose that the function $f_{Z_1}: [0, \infty) \to [0, \infty)$ is continuously differentiable everywhere, and that
	\[
	\int_{[0, \infty)} \frac{f_{Z_1}'(z)^2}{f_{Z_1}(z)} \,\mathrm{d}z < \infty.
	\]
	
	Additionally, assume that 
	\begin{equation}\label{eq:cond:queue}
	\liminf_{n\to\infty}\frac{1}{n}\sum_{k=0}^{n-1}\P (S_k > Z_1)^2 > 0,
	\end{equation}
	and that for some exponent $\delta > 0$, the series of mixing coefficients $(\alpha^S(n))_{n \in \N}$ satisfies
	\[
	\sum_{n \geq 1} n^{\delta} \alpha^S(n) < \infty.
	\]
	
	Let $\Sigma_n = W_1 + \ldots + W_n$, $v_n(t) = \min \{ 1 \leq k \leq n \mid \var(\Sigma_k) \geq t \var(\Sigma_n) \}$, and for $1 \leq k \leq n$, define
	\[
	\xi_{k,n} = \frac{W_k - \E [W_k]}{\var(\Sigma_n)^{1/2}}.
	\]
	
	Then the sequence of functions $B_n(t) = \sum_{k=1}^{v_n(t)} \xi_{k,n}$, for $t \in (0, 1]$ and $n \geq 1$, converges in distribution in $D([0, 1])$ (equipped with the uniform topology) to a standard Brownian motion $B$.
\end{theorem}
\begin{proof}
	In the proof of Theorem \ref{thm:queue:main}, we showed that under these conditions, the sequence $(W_n)_{n \in \N}$ forms a Markov chain in a random environment, which satisfies the conditions of Theorem \ref{thm:MCRE_LLN_and_weakapproaching}. Additionally, Lemma \ref{lem:q:fclt}, in conjunction with the condition \eqref{eq:cond:queue}, guarantees that 
	\[
	\liminf_{n\to\infty}\frac{1}{n}\sum_{k=0}^{n-1}\E \left[ \frac{1}{r(I_k)} \lVert \partial_{s_k} \E \left[ S_n \mid \ul{S} \right] \rVert^2 \right]>0.
	\] 
	
	With the Lyapunov function $V(w) = e^{\bar{t}w - 1}$, the inequality \eqref{eq:Phi} that appears in the conditions of Theorem \ref{thm:MCRE_LLN_and_weakapproaching} holds for  $\Phi = \mathrm{id}_{\mathbb{R}^+}$ and any $p > 1$. Consequently, the sequence $(\alpha^S(n))_{n \in \mathbb{N}}$ satisfies the condition regarding the mixing coefficients in Theorem \ref{thm:MCRE_FCLT}. 
	
	Overall, we conclude that the conditions of Theorem \ref{thm:MCRE_FCLT} are met, which leads to the desired result.
\end{proof}

%

\medskip 
The conditions in Assumptions~\ref{as:queue:Z} and~\ref{as:queue:S} can most likely be significantly weakened. We conjecture that the boundedness condition on the service times $(S_n)_{n \in \mathbb{N}}$ and the assumption on the tail probabilities of $Z_1$ are of a purely technical nature and may be removed in future work. Theorem~1.8 in~\cite{lovas2021ergodic}, for instance, addresses the case when both $S$ and $Z$ are unbounded, $(S_n)_{n \in \mathbb{N}}$ is i.i.d., and so is $(Z_n)_{n \ge 1}$. In that setting, the main challenge is that the condition on the minorization coefficient in part~B) of Assumption~\ref{as:dynamics} is no longer satisfied. However, the results of our paper~\cite{lovas} can still be applied under appropriate  assumptions on the environment. One such condition, as formulated in Theorem~1.8 of~\cite{lovas2021ergodic}, is that $Z_1$ has a Gumbel-like tail:
\[
\P(Z_1 \ge z) \le C_1 \exp(-C_2 e^{C_3 z}), \quad \text{for some } C_1,\, C_2,\, C_3 > 0.
\]

Answering such questions, as well as investigating the case when the service times $(S_n)_{n \in \mathbb{N}}$ are i.i.d.\ and the inter-arrival times $(Z_n)_{n \ge 1}$ are weakly dependent, will be the subject of our future research.

\subsection*{Acknowledgment}

I extend my sincere gratitude to Miklós Rásonyi for his invaluable assistance and insightful discussions on the subject matter, which greatly contributed to the development of this paper. I would also like to thank the two anonymous reviewers for their careful reading and constructive comments, which significantly improved the quality of the manuscript.

\subsection*{Declaration of generative AI and AI-assisted technologies in the writing process}

The author(s) would like to declare that during the preparation of this work, ChatGPT 3.5 was utilized to enhance the language and readability of the paper. As a non-native English speaker, the author(s) sought to improve the quality of the English language presentation. It is important to emphasize that the ideas and content of this article represent the author's original intellectual contributions. After utilizing this tool, the author(s) thoroughly reviewed and edited the content as necessary and take(s) full responsibility for the final publication.


\newpage
\appendix
\section{Brief survey of key results on $\alpha$-mixing sequences}\label{sec:mixing_survey}

In this section, we present a selection of theorems from the literature on $\alpha$-mixing sequences. These results encompass fundamental topics such as the law of large numbers, and the central limit theorem. However, we do not address all relevant topics--for instance, the distribution of extreme values in $\alpha$-mixing sequences or concentration inequalities. For a detailed treatment of these subjects, we refer the interested reader to \cite{Welsch1972} and \cite{merlevede2011concentration}.  

By integrating the theorems presented here with the transition of mixing results from Sections \ref{sec:mixtrans} and \ref{sec:MCRE}, we establish a robust theoretical framework. This framework enables the statistical analysis of non-linear autoregressive processes with exogenous covariates and Markov chains in random environments. Given the vast literature on $\alpha$-mixing processes, a comprehensive review is beyond the scope of this paper. Instead, we focus on key results that are particularly useful for the statistical analysis of weakly dependent sequences. For readers seeking a more in-depth overview, we recommend Doukhan \cite{doukhan94}, which, with its detailed references and literature survey, provides an excellent starting point.

In theorems concerning $\alpha$-mixing sequences, the key condition is typically about how rapidly the mixing coefficient sequence $(\alpha^W(n))_{n\in\N}$ decays to zero. In the econometric literature, the term \emph{size} is frequently used to characterize this behavior (cf. Definition 3.45 in \cite{white2000}). However, since the definition is not uniform, despite its ability to make theorems more concise and elegant, we avoid using size to prevent misunderstandings.

In the context of stationary processes, it is well-known that strong mixing implies ergodicity ensuring the applicability of the strong law of large numbers. 
For non-stationary, heterogeneously distributed $\alpha$-mixing sequences, McLeish \cite{McLeish} established the following version of the strong law of large numbers.
\begin{theorem}[McLeish, 1975]\label{thm:McLeish}
	Consider a sequence of $\R$-valued  random variables $(W_n)_{n\in\N}$ with $\E W_n=0$, $n\in\N$ and with $\alpha$-mixing coefficients satisfying
	$\alpha^W (n)\le c n^{-\frac{r}{r-2}}$, $n\in\N$,
	for $c>0$ and $r>2$. Suppose that for some $p$ such that $r/2<p\le r<\infty$,
	$$
	\sum_{n=1}^\infty \frac{\E^{2/r} |W_n|^p}{n^{2p/r}}<\infty.
	$$
	Under these conditions,
	$$
	\frac{1}{n}\sum_{k=1}^n W_n\stackrel{\Pas}{\to} 0,\quad n\to\infty.
	$$
\end{theorem}
\begin{remark}\label{rem:McLeish}
	The condition related to the decay of the moments in Theorem \ref{thm:McLeish}
	$$
	\sum_{n=1}^\infty \frac{\E^{2/r} |W_n|^p}{n^{2p/r}} < \infty
	$$
	is automatically satisfied if $\sup_{n \in \N} \E |W_n|^p < \infty$.
\end{remark}

In the context of ergodic theory, we distinguish between ergodic, weakly mixing, and strongly mixing processes. 
We classify a collection of random variables of the form 
$\{\xi_{n,i} \mid 1 \le i \le n\}$ as weakly mixing if
$$
\frac{1}{n} \sum_{k=1}^{n} \alpha^{\xi_{n, \cdot}}(k) \to 0, \quad n \to \infty,
$$
where $\alpha^{\xi_{n, \cdot}}(\cdot)$ denotes the $\alpha$-mixing coefficient corresponding to the $n$th row of the array.

Hansen established the weak law of large numbers, more precisely $L^1$-law of large numbers, for heterogenous weak mixing processes and arrays \cite{hansen2019weak}. A notable and valuable characteristic of this result is the absence of a stationarity assumption, thereby broadening the applicability of the result. The elegance of Hansen's proof lies in its simplicity, leveraging the standard representation of the variance of the truncated mean as the weighted Ces\`{a}ro sum of covariances, and bounds the latter using the mixing inequality for bounded random variables. 

As noted by Hansen \cite{hansen2019weak}, strong mixing implies weak mixing, and weak mixing, in turn, implies ergodicity, moreover this nesting is strict. However, for any process $(W_n)_{n\in\N}$, the sequence of mixing coefficients $(\alpha^W(n))_{n\in\N}$ is monotonically decreasing. Consequently, their Ces\`{a}ro sum converges to zero if and only if $\alpha^W(n) \to 0$ as $n \to \infty$. As a result, the classes of weakly mixing and strongly mixing random sequences coincide. Thus, for weakly dependent sequences, Hansen's theorem can be stated as follows:
\begin{theorem}[Hansen, 2019]\label{thm:Hansen}
	Consider a strongly mixing $\R$-valued process $(W_n)_{n\in\N}$, and define the sequence of partial sums $S_n:=\sum_{k=1}^{n} W_k$, $n\ge 1$. Additionally, suppose that the condition
	\begin{equation}\label{eq:uint}
		\lim_{B\to\infty}\sup_{n\ge 1}\frac{1}{n}\sum_{k=1}^{n}
		\E (|W_k|\ind (|W_k|\ge B)) = 0
	\end{equation}
	holds.
	Then, we have 
	$$
	\frac{S_n}{n}-\frac{\E (S_n)}{n}\stackrel{L^1}{\to} 0,\,\,n\to\infty.
	$$
\end{theorem}
\begin{proof}
	Hansen stated and proved this theorem, in a bit more general setting, for triangular arrays of weakly mixing variables (See Theorem 1 on page 4 in \cite{hansen2019weak}). For making the explanation self-contained, we present the proof of this simpler version here.

	Without the loss of generality we can assume that $\E (W_n)=0$, $n\in\N$. Let $\eps>0$ be arbitrary and choose $B>0$ such that
	\begin{equation}\label{eq:uint2}
		\sup_n\frac{1}{n}\sum_{k=1}^n \E (|W_k|\ind (|W_k|>B))<\eps.
	\end{equation}
	Let us introduce 
	\begin{align*}
		W_n'  &= W_n\ind (|W_n|\le B)-\E (W_n\ind (|W_n|\le B)) \\
		W_n'' &= W_n\ind (|W_n|> B)-\E (W_n\ind (|W_n|> B)).	
	\end{align*}
	Obviously, $W_n = W_n'+W_n''$, $n\in\N$ hence by the triangle inequality and \eqref{eq:uint2}, we have
	\begin{equation}\label{eq:hansen_L1_est}
		\frac{1}{n}\E\left|\sum_{k=1}^n W_k\right|
		\le
		\frac{1}{n}\E\left|\sum_{k=1}^n W_k'\right|+
		\frac{2}{n}\sum_{k=1}^n \E (|W_k|\ind (|W_k|>B))
		<\frac{1}{n}\E\left|\sum_{k=1}^n W_k'\right|+2\eps.
	\end{equation}
	Furthermore, $W_n'$ satisfies the bound $|W_n'|\le 2B$, and for its mixing coefficient
	$\alpha^{W'}(n)\le \alpha^W (n)$, $n\in\N$ holds, consequently by the mixing inequality
	for bounded variables (cf. Theorem A.5 in \cite{hall2014martingale} or the proof of Lemma A.1. in \cite{lovasCLT}), 
	$$
	|\E (W_k' W_l')|=|\cov (W_k',W_l')|\le 16 B^2 \alpha^W (|k-l|). 
	$$
	By Jensen's inequality, we can estimate
	\begin{align*}
		\E^2\left|\sum_{k=1}^n W_k'\right| 
		\le
		\sum_{k,l=1}^n \E (W_k' W_l')
		\le 
		16 B^2 
		\sum_{k,l=1}^n
		\alpha^W (|k-l|)=
		16B^2 n
		\left(
		\alpha^W (0)
		+
		2\sum_{k=1}^{n}
		\alpha^W (k)
		\right).
	\end{align*}
	Substituting this into \eqref{eq:hansen_L1_est} yields
	\begin{align*}
		\frac{1}{n}\E\left|\sum_{k=1}^n W_k\right|
		<
		4B \left(
		\frac{\alpha^W (0)}{n}
		+
		\frac{2}{n}\sum_{k=1}^{n}\alpha^W (k) 
		\right)^{1/2} + 2\eps,
	\end{align*}
	where the upper bound tends to $2\eps$ as $n\to\infty$ since $(W_n)_{n\in\N}$ is strongly mixing,
	and thus $\limsup_{n\to\infty}\frac{1}{n}\E\left|\sum_{k=1}^n W_k\right|<2\eps$ holds for arbitrary $\eps>0$ which completes the proof.	
\end{proof}

\begin{remark}\label{rem:Hansen}
	The average uniform integrability condition \eqref{eq:uint} is automatically satisfied if the process $(W_n)_{n\in\N}$ has a uniformly bounded moment $\sup_{n\in\N}\E (|W_n|^r)<\infty$
	for some $r>1$.
\end{remark}

Our objective is to establish the (functional) central limit theorem for certain functionals of the sequence of iterates $(X_n)_{n\in\N}$ when $(Y_n)_{n\in\N}$ is merely $\alpha$-mixing and stationarity is not assumed. While we studied this problem in the context of stochastic gradient Langevin dynamics \cite{lovasCLT}, our focus was limited to stationary data streams. Our approach relied significantly on Corollary 2 in \cite{herrndorf1984}, which offers broad applicability, extending even to non-stationary processes. However, the condition $\lim_{n\to\infty} n^{-1}\mathbb{E}(S_n^2) = \sigma^2$ required by this corollary is not generally met in cases  when the exogenous regressor $(Y_n)_{n\in\mathbb{N}}$ is non-stationary.

Very recently there was major progress on the functional CLT for non-stationary mixing sequences.
In \cite{merlevede2019functional} Merlev{\`e}de, Peligrad and Utev answered the question raised by Ibragimov concerning the CLT for triangular arrays of non-stationary weakly dependent variables under the Lindeberg condition (cf. page 1 in \cite{hafouta2021functional}). Subsequently, Merlev{\`e}de and Peligrad proved the functional CLT for triangular arrays satisfying a dependence condition weaker than the standard strong mixing condition, termed the weak strong mixing condition \cite{merlevede2020functional}. Both of these results require the condition:
\begin{equation}\label{eq:varSn}
	\sum_{k=1}^{n} \var(W_k) = O(\var(S_n))
\end{equation}
which is difficult to verify in general. A key contribution of our paper is the derivation of a Cramér-Rao lower bound for the variance of partial sums in \eqref{eq:varSn}, facilitating a functional CLT when $(X_n)_{n\in\N}$ forms a Markov chain in a random environment (See Section \ref{sec:MCRE}).

Essentially, aside from certain moment conditions, Theorem 1 in Ekstr\"om's paper \cite{ekstrom2014general} mandates only the verification that the $\alpha$-mixing coefficients exhibit a sufficiently rapid decrease:

\begin{theorem}[Ekstr\"om, 2014]\label{thm:Ekstrom}
	Let $\{\xi_{n,i}\mid 1\le i\le d_n,\,\,n\in\N\}$ be an array of $\R$-valued random variables with $\E \xi_{n,i} = 0$, $1\le i\le d_n,\,\,n\in\N$, and define $S_n=\sum_{k=1}^{d_n} \xi_{n,k}$, $n\in\N$.
	Assume that for some $r>0$, 
	\begin{enumerate}[i.]
		\item $\displaystyle{\sup_{n\in\N}\max_{1\le i\le d_n} \E |\xi_{n,i}|^{2+r}<\infty}$, and
		
		\item $\sup_{n\in\N}\sum_{k=0}^{\infty}(k+1)^2 \left(\alpha^{\xi_{n,\cdot}} (k)\right)^{\frac{r}{4+r}}<\infty$,
	\end{enumerate} 
	where $\alpha^{\xi_{n,\cdot}} (\cdot)$ denotes the strong mixing coefficient corresponding to the $n$th row in the array.
	
	Then the distributions $\law{d_n^{-1/2}S_n}$ and $\mathcal{N}(0,\var (d_n^{-1/2}S_n))$ are \emph{weakly approaching}, that is for any bounded continuous function $g:\R\to\R$,
	$$
	\E \left[g\left(d_n^{-1/2}S_n\right)\right] - 
	\int_{\R} g\left(\var (d_n^{-1/2}S_n)^{1/2}t\right)\frac{1}{\sqrt{2\pi}}e^{-\frac{t^2}{2}}\,\dint t
	\to 0\,\,\text{as}\,\,n\to\infty.
	$$
\end{theorem}
\begin{proof}
	For the proof, we refer the reader to \cite{ekstrom2014general}.
\end{proof}

\begin{corollary}\label{cor:ekstrom}
	Let $(W_n)_{n\in\N}$ be a sequence of $\R$-valued zero mean random variables.
	Suppose there exists $r>0$ such that $\sup_{n\in\N}\E |W_n|^{2+r}<\infty$ and $\sum_{k=0}^{\infty}(k+1)^2 \left(\alpha^{W} (k)\right)^{\frac{r}{4+r}}<\infty$ holds.
	Under these conditions, the distributions of $n^{-1/2}S_n$ and $\mathcal{N}(0,\var(n^{-1/2}S_n))$ are weakly approaching.
	
	Moreover, if the sequence $(n/\sigma_n^2)_{n\ge 1}$ is bounded, where $\sigma_n^2=\E S_n^2$, $n\in\N$, then $\law{S_n/\sigma_n}$ converges weakly to the standard normal distribution.
\end{corollary}
\begin{proof}
	The first part of the statement is immediately follows from Theorem \ref{thm:Ekstrom} with
	$\xi_{n,i}=\frac{\sqrt{n}}{\sigma_n}W_i$, $1\le i\le d_n=n$, $n\ge 1$, since $\xi_{n,\cdot}$ and $W$ have the same $\alpha$-mixing coefficient for every $n$.
	
	As for the second part, let $\xi_{n,i}=\frac{\sqrt{n}}{\sigma_n}W_i$, $1\le i\le n$, $n\ge 1$. Easily seen that $\max_{1\le i\le n}\E |\xi_{n,i}|^{2+r}\le (n/\sigma_n^2)^{1+r/2}\E |W_i|^{2+r}$ hence by Theorem \ref{thm:Ekstrom}, for any bounded continuous function $g:\R\to\R$,
	$$
	\E \left[g\left(n^{-1/2}\sum_{k=1}^{n}\xi_{n,k}\right)\right] - 
	\int_{\R} g (t)\frac{1}{\sqrt{2\pi}}e^{-\frac{t^2}{2}}\,\dint t
	\to 0,\,\,\text{as}\,\,n\to\infty,
	$$
	which completes the proof.
\end{proof}

The next important remark in its original form can be found in Ekstr\"om's paper (c.f. Remark 1 on page 1 in \cite{ekstrom2014general}).
\begin{remark}\label{rem:ekstromremark}
	Assume that for $r>2$, $\sup_{n\in\N}\E |W_n|^{r}<\infty$ and $\sum_{k=0}^{\infty} \left(\alpha^{W} (k)\right)^{1-\frac{2}{r}}<\infty$. Then the variance of $n^{-1/2}S_n$ is bounded. 
\end{remark}

\section{Counterexample to long-term contractivity condition}\label{ap:Felsmann}

In this point we present an example for a stationary stochastic process $(Y_n)_{n\in\N}$ and a function $\ga:\Y\to (0,\infty)$, outlined in Balázs Felsmann's Master's thesis, where $\E (\ga (Y_0))<1$, and despite the favorable mixing properties of $(Y_n)_{n\in\N}$, the long-term contractivity condition \eqref{eq:LT} fails to hold.

Let $(Z_n)_{n\in\Z}$ be a sequence of i.i.d. Bernoulli variables with $\P (Z_0=0)=\P (Z_0=1)=1/2$, and define the process 
$$
Y_n = Z_n + Z_{n-1},\,\, n\in\Z,
$$ 
which takes its values in $\Y=\{0,1,2\}$. Clearly, $(Y_n)_{n\in\N}$ is a stationary process such that $Y_n$ and $Y_m$ are independent for $|m-n|>1$ hence $\alpha^Y (n)=0$ for $n\ge 2$. We consider a function $\ga:\Y\to (0,1)$, where $\ga (i)=\ga_i$, $i=0,1,2$ will be specified later.
Let us introduce the following sequences. 
\begin{equation}\label{eq:anbncn}
	a_n = \E\left(\prod_{k=1}^{n}\ga (Y_k)\right),\,\,
	b_n = \E\left(\ind_{\{Z_n=0\}}\prod_{k=1}^{n}\ga (Y_k)\right),\,\,
	c_n = \E\left(\ind_{\{Z_n=1\}}\prod_{k=1}^{n}\ga (Y_k)\right),\,\, n\in\N.
\end{equation}
Clearly, we have $a_n=b_n+c_n$, $n\in\N$, moreover we can write
\begin{align*}
 b_{n+1} &= \E\left(\ind_{\{Z_{n+1}=0\}}\ind_{\{Z_n=0\}}\prod_{k=1}^{n+1}\ga (Y_k)\right) + \E\left(\ind_{\{Z_{n+1}=0\}}\ind_{\{Z_n=1\}}\prod_{k=1}^{n+1}\ga (Y_k)\right) \\
 &=\frac{1}{2}(\ga_0 b_n+\ga_1 c_n),
\end{align*}
and similarly
\begin{align*}
c_{n+1} &= \E\left(\ind_{\{Z_{n+1}=1\}}\ind_{\{Z_n=0\}}\prod_{k=1}^{n+1}\ga (Y_k)\right) + \E\left(\ind_{\{Z_{n+1}=1\}}\ind_{\{Z_n=1\}}\prod_{k=1}^{n+1}\ga (Y_k)\right) \\
&=\frac{1}{2}(\ga_1 b_n+\ga_2 c_n),
\end{align*}
hence the linear recursion
\begin{equation}\label{eq:rec_abc}
\left[\begin{array}{c}
b_{n+1}\\
c_{n+1}
\end{array}\right] = \frac{1}{2}
\Gamma
\left[\begin{array}{c}
b_{n}\\
c_{n}
\end{array}\right]
\end{equation}
holds, where $\Gamma=\left[
\begin{array}{cc}
\ga_0 & \ga_1 \\
\ga_1 & \ga_2
\end{array}
\right]$, and $b_0=c_0=1/2$. Thus for $(a_n)_{n\in\N}$, we have
$$
a_n = \frac{1}{2^{n+1}}[1,\, 1]\Gamma^n\left[\begin{array}{c}
1\\
1
\end{array}\right].
$$
By setting $\ga_0=3$, $\ga_1=\ga_2=0$, we obtain $a_0=1$ and $a_n=\frac{1}{2}(3/2)^n$, for $n\ge 1$, thus $a_n\to\infty$, as $n\to\infty$. For $\eps\in (0,1/4)$, let us define $\ga_\eps:\Y\to (0,\infty)$, $\ga_\eps (y)=\ga (y)+\eps$, $y=0,1,2$. So, we have $\E (\ga_{\eps} (Y_0))=3/4+\eps<1$, on the other hand
\begin{align*}
	\E^{1/n}\left(\prod_{k=1}^{n}\ga_\eps (Y_k)\right)\ge
	\E^{1/n}\left(\prod_{k=1}^{n}\ga (Y_k)\right)=\frac{3}{2}2^{-1/n},\,\,n\ge 1,
\end{align*}
and thus
$$
\liminf_{n\to\infty} \E^{1/n}\left(\prod_{k=1}^{n}\ga_\eps (Y_k)\right)\ge \frac{3}{2}
$$
which means that the long-term contractivity condition fails to hold in this situation.

\section{Coupling condition for MCREs}\label{ap:CouplingCondition}

In this section, our primary objective is to establish an upper bound for the non-coupling probabilities as in Definition \ref{def:coupling}. Our approach refines the proof of Lemma 3.10 in \cite{lovasCLT}. To achieve this, we introduce the following lemma, which is identical to Lemma 7.4 in \cite{lovas}. It describes the consequences of the drift condition satisfied by $Q(y_{k-1})\ldots Q(y_l)$, where $\mathbf{y}\in\Y^\N$ and $k,l\in\N$, $l<k$ are arbitrary and fixed.
\begin{lemma}\label{lem:ita} 
	For $x\in\X$, $\mathbf{y}\in\Y^\N$ and $k,l\in\N$, $l<k$, we have	
	\begin{align*}
	\left[Q(y_{k-1})\ldots Q(y_{l}) V\right](x) &\le 
	V(x)\,\prod_{r=l}^{k-1}
	\ga (y_r)
	+ \sum_{r=l}^{k-1}K(y_r)\prod_{j=r+1}^{k-1}
	\ga (y_j).
	\end{align*}
\end{lemma}
\begin{proof}	
	We proceed by induction. Let $x \in \X$ and $l \in \N$ be arbitrary but fixed. For the base case $k = l+1$, we have
	\begin{equation}
		\left[Q(y_{l})V\right](x) \le \ga (y_l)V(x) + K(y_l),
	\end{equation}
	which follows directly from the drift condition \eqref{eq:Lyapunov}.
	
	\medskip\noindent
	\emph{Induction hypothesis:} Assume that the inequality holds for some $k > l$, with fixed $x \in \X$ and $l \in \N$:
	\begin{equation}\label{eq:ind:hypo}
		\left[Q(y_{k-1})\ldots Q(y_{l}) V\right](x) \le 
		V(x)\,\prod_{r=l}^{k-1}
		\ga (y_r)
		+ \sum_{r=l}^{k-1}K(y_r)\prod_{j=r+1}^{k-1}
		\ga (y_j).
	\end{equation}
	
	\medskip\noindent
	\emph{Induction step:} We verify that inequality \eqref{eq:ind:hypo} holds for $k+1$. By the drift property for $Q(y_k)$ we have
	$$
	[Q(y_k)V](x)\le \ga (y_k)V(x)+K(y_k).
	$$
	
	Operators $V\mapsto [Q(y)V]$, $y\in\Y$ are linear, monotone and for $V\equiv 1$ $[Q(y)V]\equiv 1$, $y\in\Y$. Therefore, taking into account that successive applications of kernels are evaluated in reverse order (see the remark following Definition \ref{def:act}), we can write
	\begin{align*}
		\left[Q(y_{k})\ldots Q(y_{l}) V\right](x) &=
			\left[Q(y_{k-1})\ldots Q(y_{l}) [Q(y_{k}) V]\right](x)
			\\
			&\le \ga (y_k) \left[Q(y_{k-1})\ldots Q(y_{l}) V\right](x)+K(y_k),
	\end{align*}
	thus by applying the induction hypothesis \eqref{eq:ind:hypo}, we obtain
	\begin{align*}
		\left[Q(y_{k})\ldots Q(y_{l}) V\right](x) 
		&\le
		\ga (y_k)\left[
		V(x)\,\prod_{r=l}^{k-1}
		\ga (y_r)
		+ \sum_{r=l}^{k-1}K(y_r)\prod_{j=r+1}^{k-1}
		\ga (y_j)
		\right]+K(y_k)
		\\
		&=
		V(x)\,\prod_{r=l}^{k}
		\ga (y_r)
		+ \sum_{r=l}^{k}K(y_r)\prod_{j=r+1}^{k}
		\ga (y_j)
	\end{align*}
	which completes the proof.
\end{proof}

In our earlier papers (See Lemma 7.1 in \cite{lovas} and Lemma 3.9 in \cite{lovasCLT}), we relied on special cases of the following lemma. It is also a variant of Lemma 6.1 in \cite{rasonyi2018}, albeit in a somewhat broader context. Such representations of parametric kernels satisfying the minorization condition \eqref{eq:smallset} can be deemed standard. For the sake of completeness and to ensure self-containment of our explanation, we present and prove it in its most general form.
\begin{lemma}\label{lem:T}
	Let $R>0$ be arbitrary and suppose the parametric kernel $Q:\Y\times\X\times\B(\X)\to [0,1]$ satisfies the minorization condition given by \eqref{eq:smallset}. Then there exists a measurable mapping $T^R:\X\times\Y\times [0,1]\to\X$ such that
	$$
	Q(y,x,A)=\int_{[0,1]}\ind_{T^R (x,y,u)\in A}\,\dint u,
	$$
	for all $x\in\X$, $A\in\B(\X)$ and $y\in\Y$. 
	Furthermore, for any fixed $y\in\Y$, there exists a Borel set $U=U (y)\in \B ([0,1])$
	with Lebesgue measure
	$
	\leb_1 (U)\ge 1-\beta (R,y)
	$
	such that for $u\in U$,
	\begin{equation}\label{eq:Tconst}
	T^R (x_1,y,u) = T^R (x_2,y,u),\,\,x_1,x_2\in\stackrel{-1}{V}([0,R]).
	\end{equation}
\end{lemma}
\begin{proof}
	We proceed as in Lemma 7.1 in \cite{lovas}, following the proof of Lemma 6.1 in \cite{rasonyi2018}. The case of countable $\X$ is straightforward and thus omitted. For the uncountable case, we can assume, by the Borel isomorphism theorem, that $\X=\R$ and $\B(\X)=\B(\R)$ is the standard Borel $\sigma$-algebra of $\R$.
	
	It is easy to see that if $\beta (R,y)=0$ for some $y\in\Y$ in the minorization condition \eqref{eq:smallset}, then for any $A\in\B (\X)$ and $x\in \stackrel{-1}{V}([0,R])$, both $Q(y,x,A) \ge \ka_R (y,A)$ and $Q(y,x,\X\setminus A) \ge \ka_R (y,\X\setminus A)$ hold simultaneously. Consequently, we have
	\begin{equation*}
	Q(y,\cdot,A)\vert_{\stackrel{-1}{V}([0,R])} = \ka_R (y,A),\,\, A\in\B(\X).
	\end{equation*}
	For $x\in\stackrel{-1}{V}([0,R])$, $A\in\B (\X)=\B(\R)$, and $y\in\Y$ let
\begin{align*}
q_R(y,x,A) := \begin{cases}
\frac{1}{\beta (R, y)}\left[
Q(y,x,A)-(1-\beta (R, y))\ka_R (y,A)
\right] & \text{ if } \beta (R, y)\ne 0\\[1.2em]
0 & \text{ if } \beta (R, y)=0
\end{cases}
\end{align*}
Additionally, introduce the pseudoinverses of the corresponding cumulative distribution functions as follows:
\begin{align*}
Q^{-1}(y,x,z) &:=  \inf \{r\in\Q\mid Q(y,x,(-\infty,r])\ge z\} \\
\ka_R^{-1}(y,z) &:= \inf \{r\in\Q\mid \ka_R (y,(-\infty,r])\ge z\} \\
q_R^{-1}(y,x,z) &:=  \inf \{r\in\Q\mid q_R(y,x,(-\infty,r])\ge z\}.	
\end{align*}
	
There exists a measurable mapping $\chi:[0,1]\to [0,1]^2$ such that the pushforward measure
$\chi_*(\dint x)$ equals $\dint x\dint y$. In other words, for every Borel set $B\in\B ([0,1]^2)$, $\leb_2 (B)=\leb_1 (\stackrel{-1}{\chi}(B))$, where $\leb_k$ refers to the Lebesgue measure on $[0,1]^k$, $k=1,2$.
Finally, we define
\begin{equation*}
T^R (x,y,u) =\begin{cases}
\ind_{\chi (u)_1\ge \beta (R,y)}\ka_R^{-1}(y,\chi (u)_2)
+\ind_{\chi (u)_1<\beta (R,y)} q_R^{-1}(y,x,\chi (u)_2)
& \text{ if } V(x)\le R \\
Q^{-1}(y,x,\chi (u)_2) & \text{ if } V(x)>R.
\end{cases}
\end{equation*}
It is evident that $x\mapsto T^R (x,y,u)$ is constant on $\stackrel{-1}{V}([0,R])$ whenever $\chi (u)_1\ge \beta (R,y)$, implying \eqref{eq:Tconst} with $U = \stackrel{-1}{\chi}([\beta (R,y),1]\times [0,1])$.

Furthermore, for any fixed $r\in\R$, $y\in\Y$ and $x\in\stackrel{-1}{V}([0,R])$
by the change of variable formula and the definitions of $T^R$ and $q_R$, we can write
\begin{align*}
	\int_{[0,1]}\ind_{\{T^R (x,y,u)\le r\}}\,\dint u 
	&=
	\int_0^1 \int_0^1
	\ind_{\{\ind_{s\ge \beta (R,y)}\ka_R^{-1}(y,t)
	+\ind_{s<\beta (R,y)} q_R^{-1}(y,x,t)\le r\}} \dint s \dint t
	\\
	&=
	(1-\beta (R,y))\int_0^1 \ind_{\ka_R^{-1}(y,t)\le r}\dint t +
	\beta (R,y) \int_0^1 \ind_{q_R^{-1}(y,x,t)\le r}\dint t
	\\
	&=
	(1-\beta (R,y))\ka_R (y,(-\infty,r]) + \beta (R,y) q_R(y,x,(-\infty,r])
	\\
	&=Q(y,x, (-\infty,r]).
\end{align*}
Similarly, for $x\notin\stackrel{-1}{V}([0,R])$, we have
\begin{align*}
\int_{[0,1]}\ind_{T^R (x,y,u)\le r}\,\dint u 
&=
\int_0^1 \int_0^1 
\ind_{Q^{-1}(y,x,s)\le r}\dint s\dint t 
= \int_0^1 
\ind_{Q^{-1}(y,x,s)\le r}\dint s = Q(y,x, (-\infty,r]).
\end{align*}
thus the claimed identity holds.

It remains only to show that $T^R$ is measurable with respect to sigma algebras $\B (\R)\otimes\B (\Y)\otimes \B ([0,1])$ and $\B (\R)$. Indeed, $T^R$ is a composition of measurable functions. 
\end{proof}

By Assumption \ref{as:dynamics} B), there exists $0<r<1/\bar{\ga}-1$ such that 
$
\bar{\beta}:=\sup_{y\in\Y} \beta (R(y),y)<1
$
holds with $R(y):=\frac{2K(y)}{r\ga (y)}$.
We define the measurable mapping 
\begin{equation}\label{eq:fgood}
	(x,y,u)\mapsto f(x,y,u):=T^{R(y)}(x,y,u),\,\,
	x\in\X,\,y\in\Y,\,u\in [0,1].
\end{equation}

Let $(\eps_t)_{t\in\N}$ be a sequence of i.i.d. variables uniformly distributed on $[0,1]$ such that sigma algebras $\F^{\eps}_{0,\infty}$ and $\sigma (Y_t, X_t, t\in\N)$ are independent. 
Furthermore, for $s\in\N$ and $x\in\X$, let us introduce the family of auxiliary processes
\begin{equation}\label{eq:Zaux}
Z_{s,t}^{x,\mathbf{y}} = \begin{cases}
x &\text{ if } t\le s \\
f(Z_{s,t-1}^{x,\mathbf{y}},y_{t-1},\eps_t) &\text{ if } t>s,
\end{cases}
\end{equation}
where $\mathbf{y}=(y_0,y_1,\ldots)\in\Y^\N$ can be any fixed trajectory. Clearly, for  $\mathbf{Y}=(Y_n)_{n\in\N}$, the process $Z_{s,t}^{X_s,\mathbf{Y}}$, ${t\ge s}$ is a version of $(X_t)_{t\ge s}$. In the forthcoming part of the section, we will prove that this process satisfies the coupling condition. First, we will show that for any fixed $x\in\N$ and $\mathbf{y}\in\Y^\N$, the process $Z_{s,t}^{x, \mathbf{y}}$, $t\ge s$ is a \emph{Harris recurrent} time-inhomogeneous Markov chain. The next lemma provides a quenched version of the coupling condition, controlling the coupling time between iterations starting from different initial values.

\begin{lemma}\label{lem:coupling}
	Let $x_1,x_2\in\X$ and $\mathbf{y}\in\Y^\N$ be arbitrary but fixed. Then for $0<m<n$, we have
	$$
	\P (Z_{0,n}^{x_1,\mathbf{y}} \ne Z_{0,n}^{x_2,\mathbf{y}})
	\le
	\bar{\beta}^m
	+
	(1+r)^{\lfrf{\frac{n-1}{m}}}\left\{
	\frac{V(x_1)+V(x_2)}{2}\prod_{j=0}^{\lfrf{\frac{n-1}{m}}}\ga (y_j)
	+
	\sum_{k=0}^{m-1}\sum_{l=k}^{k\lfrf{\frac{n-1}{m}}} K(y_l)\prod_{j=1}^{\lfrf{\frac{n-1}{m}}}\ga (y_{l+j})
	\right\}.
	$$
\end{lemma}
\begin{proof}
	For the sake of brevity, we employ a more concise notation: $Z_n^i=Z_{0,n}^{x_i,\mathbf{y}}$, $i=1,2$, and $n\in\N$, moreover we introduce
	$$
	\overline{Z}_{n}:=\left(Z_n^1,Z_n^2\right), \,\,
	\lVert \overline{Z}_{n}\rVert:=\max \left( V(Z_n^1), V(Z_n^2) \right)
	$$ 
	and the sequence of successive visiting times
	$$
	\sigma_0:=0,\,\sigma_{k+1} = \inf\left\{n>\sigma_k\middle|  \lVert \overline{Z}_n\rVert\le 
	R(y_n)
	\right\},\,k\in\N
	$$
	that are obviously $\F_{1,\infty}^\eps$-stopping times. Note that on 
	$\{ \lVert \overline{Z}_n\rVert>R (y_n)\}$ we have 
	\begin{equation}\label{eq:contr}
	\ga (y_n) (V( Z_n^1)+V( Z_n^2))+2K(y_n)
	\le
	(1+r)\ga (y_n) (V( Z_n^1)+V( Z_n^2))
	\end{equation}
	and thus for $k\ge 1$ and $s\ge 1$, by the Markov inequality, we obtain
	\begin{align*}
	\P (\sigma_{k+1}-\sigma_{k}>s\mid \F_{1,\sigma_k}^\eps)
	&\le 
	\E \left(
	\P (C_{\sigma_k+s}\mid \overline{Z}_{\sigma_{k}+s-1})
	\prod_{j=1}^{s-1}
	\ind_{C_{\sigma_k+j}}
	\middle| \F_{1,\sigma_k}^\eps\right)\\
	&\le 
	\frac{(1+r)\ga (y_{\sigma_k+s-1})}{R(y_{\sigma_k+s})}
	\E \left(
	\left(V( Z_{\sigma_k+s-1}^1)+V( Z_{\sigma_k+s-1}^2)\right)
	\prod_{j=1}^{s-2}
	\ind_{C_{\sigma_k+j}}
	\middle| \F_{1,\sigma_k}^\eps\right),
	\end{align*}
	where $C_n$ is the shorthand notation for the event $\{\lVert \overline{Z}_n\rVert> 
	R(y_n)\}$, $n\in\N$. 
	
	By the tower rule, we can write
	\begin{multline*}
		\E \left(
		\left(V( Z_{\sigma_k+s-1}^1)+V( Z_{\sigma_k+s-1}^2)\right)
		\prod_{j=1}^{s-2}
		\ind_{C_{\sigma_k+j}}
		\middle| \F_{1,\sigma_k}^\eps\right)
		=\\
		\E \left(
		\E\left[V( Z_{\sigma_k+s-1}^1)+V( Z_{\sigma_k+s-1}^2)\middle|\F_{1,\sigma_{k}+s-2}^\eps\right]
		\prod_{j=1}^{s-2}
		\ind_{C_{\sigma_k+j}}
		\middle| \F_{1,\sigma_k}^\eps\right).
	\end{multline*}
	
	Using the Markov property of $(\bar{Z}_n)_{n\in\N}$ and the drift property of the parametric kernel \eqref{eq:Lyapunov}, we have
	\begin{align*}
		\E\left[V( Z_{\sigma_k+s-1}^1)+V( Z_{\sigma_k+s-1}^2)\middle|\F_{1,\sigma_{k}+s-2}^\eps\right]
		&=
		\E\left[V( Z_{\sigma_k+s-1}^1)+V( Z_{\sigma_k+s-1}^2)\middle|\bar{Z}_{\sigma_{k}+s-2}\right]
		\\
		&=
		[Q(y_{\sigma_k+s-2})V](Z_{\sigma_k+s-2}^1)+
		[Q(y_{\sigma_k+s-2})V](Z_{\sigma_k+s-2}^2)
		\\
		&\le
		\ga (y_{\sigma_k+s-2})\left(V(Z_{\sigma_k+s-2}^1) + (Z_{\sigma_k+s-2}^2)\right)+2K(y_{\sigma_k+s-2}).
	\end{align*}
	Now, inequality \eqref{eq:contr} yields
	$$
	\E\left[V( Z_{\sigma_k+s-1}^1)+V( Z_{\sigma_k+s-1}^2)\middle|\F_{1,\sigma_{k}+s-2}^\eps\right]\ind_{C_{\sigma_k+s-2}}\le (1+r)\ga (y_{\sigma_k+s-2})\left(V(Z_{\sigma_k+s-2}^1) + (Z_{\sigma_k+s-2}^2)\right).
	$$
	Finally, we arrive at
	\begin{multline*}
			\E \left(
		\left(V( Z_{\sigma_k+s-1}^1)+V( Z_{\sigma_k+s-1}^2)\right)
		\prod_{j=1}^{s-2}
		\ind_{C_{\sigma_k+j}}
		\middle| \F_{1,\sigma_k}^\eps\right)
		\le 
		\\
		(1+r)\ga (y_{\sigma_k+s-2})
		\E \left(
		\left(V( Z_{\sigma_k+s-2}^1)+V( Z_{\sigma_k+s-2}^2)\right)
		\prod_{j=1}^{s-3}
		\ind_{C_{\sigma_k+j}}
		\middle| \F_{1,\sigma_k}^\eps\right).
	\end{multline*}
	
	Iteration of this argument in $s-2$ steps leads to the following estimation:
	\begin{align*}
	\P (\sigma_{k+1}-\sigma_{k}>s\mid \F_{1,\sigma_k}^\eps) 
	&\le  
	\frac{(1+r)^{s-1}\prod_{j=1}^{s-1}\ga (y_{\sigma_k+j})}{R(y_{\sigma_k+s})}
	\E \left(
	V( Z_{\sigma_k+1}^1)+V( Z_{\sigma_k+1}^2)
	\middle| \overline{Z}_{\sigma_k}\right) \\
	&\le
	\frac{r (1+r)^{s-1}\prod_{j=1}^{s}\ga (y_{\sigma_k+j})}{2K(y_{\sigma_k+s})}
	\left[\ga (y_{\sigma_k})\left(V( Z_{\sigma_k}^1)+V( Z_{\sigma_k}^2)\right) +2K(y_{\sigma_k})\right] \\
	&\le (1+r)^s K(y_{\sigma_k})\prod_{j=1}^{s}\ga (y_{\sigma_k+j}),
	\end{align*}
	where we used that $V( Z_{\sigma_k}^1)+V( Z_{\sigma_k}^2)\le 2R(y_{\sigma_k})$, and $K(\cdot)\ge 1$.
	
	Along similar lines, we can show that 
	\begin{align*}
		\P (\sigma_1>s) &\le \frac{(1+r)^{s-1}\prod_{j=1}^{s-1}\ga (y_j)}{R(y_s)}
		\left[\ga (y_0)\left(V(x_1)+V(x_2)\right)+2K(y_0)\right]
		\\
		&\le
		(1+r)^s \prod_{j=1}^{s}\ga (y_j)\left[\frac{\ga (y_0)}{2}\left(V(x_1)+V(x_2)\right)+K(y_0)\right].
	\end{align*}
	
	Clearly, for any $0<m\le n$, on the event $\{\sigma_m>n\}$ we have $\{\sigma_{k+1}-\sigma_{k}>\lfrf{n/m}\}\cap\{\sigma_{k}\le k\lfrf{n/m}\}$ for some $k=0,1,\ldots m-1$ hence by the union bound and the estimates we obtain for the time elapsed between consecutive visits, we can write
	\begin{align*}
		\P (\sigma_m>n)&\le \P \left(\bigcup_{k=0}^{m-1} \{\sigma_{k+1}-\sigma_{k}>\lfrf{n/m}\}\cap\{\sigma_{k}\le k\lfrf{n/m}\}\right)
		\\
		&\le 
		\sum_{k=0}^{m-1}\P \left(\{\sigma_{k+1}-\sigma_{k}>\lfrf{n/m}\}\cap\{\sigma_{k}\le k\lfrf{n/m}\}\right)
		\\
		&=
		\sum_{k=0}^{m-1}\sum_{l=k}^{k\lfrf{n/m}}
		\P \left(\sigma_{k+1}-\sigma_{k}>\lfrf{n/m}\mid \sigma_{k}=l\right)\P(\sigma_{k}=l)
		\\
		&\le 
		\P(\sigma_{1}>\lfrf{n/m})+
			\sum_{k=1}^{m-1}\sum_{l=k}^{k\lfrf{n/m}}
		\P \left(\sigma_{k+1}-\sigma_{k}>\lfrf{n/m}\mid \sigma_{k}=l\right)
	\\
	&\le 
	\left[\frac{\ga (y_0)}{2}\left(V(x_1)+V(x_2)\right)+K(y_0)\right]
	(1+r)^{\lfrf{n/m}} \prod_{j=1}^{\lfrf{n/m}}\ga (y_j)
	\\
	&+
	(1+r)^{\lfrf{n/m}}\sum_{k=1}^{m-1}\sum_{l=k}^{k\lfrf{n/m}} K(y_l)\prod_{j=1}^{\lfrf{n/m}}\ga (y_{l+j}).
	\end{align*}
	
	Next, we estimate the probability of no-coupling on events when small sets are visited at least $m$-times. By Lemma \ref{lem:T} and the choice of $R(y)$ in the definition of $f$ in \eqref{eq:fgood} and \eqref{eq:Zaux}, for each $j=1,\ldots,m$, $x \mapsto f(x,y_j,\eps_{\sigma_j+1})$ is constant on the level set $\stackrel{-1}{V}([0,R(y_{\sigma_j})])$ with probability at least $1-\bar{\beta}$ 
	independently of $\F_{0,\sigma_j}^\eps$ thus no-coupling happens with probability at most $\bar{\beta}$.
	Therefore, we can estimate:
	\begin{align*}
	\P (Z_{0,n}^{x_1,\mathbf{y}}\ne Z_{0,n}^{x_2,\mathbf{y}},\sigma_{m}<n) 
	\le \bar{\beta}^m. 
	\end{align*}
	
	Finally, we combine this upper bound with that one what we got for the tail probability of the visiting times, and obtain
	\begin{align*}
	&\P (Z_{0,n}^{x_1,\mathbf{y}}\ne Z_{0,n}^{x_2,\mathbf{y}})
	\le 
	\P (Z_{0,n}^{x_1,\mathbf{y}}\ne Z_{0,n}^{x_2,\mathbf{y}},\sigma_{m}<n)
	+
	\P (\sigma_{m}> n-1)\\
	& \le
	\bar{\beta}^m
	+
	\left[\frac{\ga (y_0)}{2}\left(V(x_1)+V(x_2)\right)+K(y_0)\right]
	(1+r)^{\lfrf{\frac{n-1}{m}}} \prod_{j=1}^{\lfrf{\frac{n-1}{m}}}\ga (y_j)
	\\
	&+
	(1+r)^{\lfrf{\frac{n-1}{m}}}\sum_{k=1}^{m-1}\sum_{l=k}^{k\lfrf{\frac{n-1}{m}}} K(y_l)\prod_{j=1}^{\lfrf{\frac{n-1}{m}}}\ga (y_{l+j})
	\\
	&=
	\bar{\beta}^m
	+
	(1+r)^{\lfrf{\frac{n-1}{m}}}\left\{
	\frac{V(x_1)+V(x_2)}{2}\prod_{j=0}^{\lfrf{\frac{n-1}{m}}}\ga (y_j)
	+
	\sum_{k=0}^{m-1}\sum_{l=k}^{k\lfrf{\frac{n-1}{m}}} K(y_l)\prod_{j=1}^{\lfrf{\frac{n-1}{m}}}\ga (y_{l+j})
	\right\}
	\end{align*}
	which completes the proof.
	
\end{proof}

\begin{proof}[Proof of Lemma \ref{lem:MCREmixing}]
The process $Z_{0,n}^{X_0,\mathbf{Y}}$, $n\in\N$ is a version of $(X_n)_{n\in\N}$ thus 
to simplify the notation in this proof, we may and will redefine $(X_n)_{n\in\N}$ as $X_n:=Z_{0,n}^{X_0,\mathbf{Y}}$.
Furthermore, we indicate the dependence of $Z$ on the driving noise by writing $Z_{s,t}^{x,\mathbf{y},\eps}$ instead of $Z_{s,t}^{x,\mathbf{y}}$, where $\eps$ refers to $\mathbf{\eps}=(\eps_1,\eps_2,\ldots)$. We also introduce the usual left shift operation: $(S\mathbf{y})_j=y_{j+1}$ and similarly $(S\mathbf{\eps})_j=\eps_{j+1}$, $j\in\N$.

\medskip
For arbitrary but fixed $j\in\N$ and $x\in\X$, we can write
\begin{align*}
	\P (Z_{j,j+n}^{X_j,\mathbf{Y},\eps}\ne Z_{j,j+n}^{x,\mathbf{Y},\eps}) 
	=
	\P (Z_{0,n}^{X_j,S^j\mathbf{Y},S^j\eps}\ne Z_{0,n}^{x,S^j\mathbf{Y},S^j\eps})
\end{align*} 
thus by the tower rule and Lemma \ref{lem:coupling}, for $0<m<n$, we have
\begin{align}\label{eq:est1}
\begin{split}
\P (Z_{0,n}^{X_j,S^j\mathbf{Y},S^j\eps}\ne Z_{0,n}^{x,S^j\mathbf{Y},S^j\eps}) =
\E \left[\P \left(Z_{0,n}^{X_j,S^j\mathbf{Y},S^j\eps}\ne Z_{0,n}^{x,S^j\mathbf{Y},S^j\eps} \middle| 
\F_{0,\infty}^Y\vee\sigma (X_j)
\right)\right] \le \bar{\beta}^m
+& \\
(1+r)^{\lfrf{\frac{n-1}{m}}}\E\left[
\frac{\E \left[V(X_j)\mid \F_{0,\infty}^Y\right]+V(x)}{2}\prod_{r=0}^{\lfrf{\frac{n-1}{m}}}\ga (Y_{r+j})
+
\sum_{k=0}^{m-1}\sum_{l=k}^{k\lfrf{\frac{n-1}{m}}} K(Y_{l+j})\prod_{r=1}^{\lfrf{\frac{n-1}{m}}}\ga (Y_{l+r+j})
\right].
\end{split}
\end{align}
As for $\E \left[V(X_j)\mid \F_{0,\infty}^Y\right]$, by Lemma \ref{lem:ita}, we obtain
\begin{align}\label{eq:est2}
\begin{split}
\E \left[V(X_j)\mid \F_{0,\infty}^Y\right] &= \E \left[[Q(Y_{j-1}),\ldots, Q(Y_0)V](X_0)\mid \F_{0,\infty}^Y\right]
\\
&\le
	\E(V(X_0))\,\prod_{i=0}^{j-1}
\ga (Y_i)
+ \sum_{r=0}^{j-1}K(Y_r)\prod_{i=r+1}^{j-1}
\ga (Y_i),
\end{split}
\end{align}
where we used that the initial state $X_0$ is independent of $\sigma (Y_n,\eps_{n+1}\mid n\in\N)$. 

For $n>1$, we choose $m_n=\frac{n-1}{\lfrf{n^{1/2}}-1}$. It is easy to check that $m_n\ge \lfrf{n^{1/2}}$ for $n>1$ hence $\bar{\beta}^{m_n}\le \bar{\beta}^{\lfrf{n^{1/2}}}$. 
Furthermore, let us fix $\ga'>\bar{\ga}$ such that $\tilde{\ga}:=(1+r)\ga'<1$. Then 
by Assumption \ref{as:dynamics} A), there exists $N\in\N$ such that
$$
\sup_{j\ge -1} \E\left[
K(Y_j)
\prod_{k=1}^{n}
\ga (Y_{k+j})
\right] \le (\ga')^{n},\,\,n\ge N,
$$
where for our convenience, we employ the convention $K(Y_{-1}):=1$.
Using \eqref{eq:est2}, for $\lfrf{n^{1/2}}\ge N$, we obtain
\begin{align*}
\E\left[
\E \left[V(X_j)\mid \F_{0,\infty}^Y\right]\prod_{r=0}^{\lfrf{n^{1/2}}-1}\ga (Y_{r+j})
\right]	
&\le 
\E (V(X_0))\E\left[\prod_{i=0}^{\lfrf{n^{1/2}}+j-1}\ga (Y_{i})\right]
+
\sum_{r=0}^{j-1}\E
\left[
K(Y_r)\prod_{i=r+1}^{\lfrf{n^{1/2}}+j-1}\ga (Y_{i})
\right]
\\
&\le
\E (V(X_0))(\ga')^{\lfrf{n^{1/2}}+j} +
\sum_{r=0}^{j-1} (\ga')^{\lfrf{n^{1/2}}+j-1-r}
\\
&\le 
(\ga')^{\lfrf{n^{1/2}}}
\left(\E (V(X_0)) + \frac{1}{1-\ga'}\right).
\end{align*}
Similarly, for $\lfrf{n^{1/2}}-1\ge N$, we get
\begin{align*}
	\E \left[V(x)\prod_{r=0}^{\lfrf{n^{1/2}}-1} \ga (Y_{r+j})\right]
	\le (\ga')^{\lfrf{n^{1/2}}} V(x),
\end{align*}
moreover
\begin{align*}
	\sum_{k=0}^{m_n-1}\sum_{l=k}^{k(\lfrf{n^{1/2}}-1)}
	\E\left[
	 K(Y_{l+j})\prod_{r=1}^{\lfrf{n^{1/2}}-1}\ga (Y_{l+r+j})
	\right] 
	&\le 
	\sum_{k=0}^{m_n-1}\sum_{l=k}^{k(\lfrf{n^{1/2}}-1)} (\ga')^{\lfrf{n^{1/2}}-1}
	\\
	&=
	m_n\left[\frac{(m_n-1)(\lfrf{n^{1/2}}-2)}{2}+1\right]
	(\ga')^{\lfrf{n^{1/2}}-1}.
\end{align*}

To sum up, and taking into account that $m_n=O(n^{1/2})$, we obtain 
\begin{align*}
\P (Z_{j,j+n}^{X_j,\mathbf{Y},\eps}\ne Z_{j,j+n}^{x,\mathbf{Y},\eps})
&\le
\bar{\beta}^{n^{1/2}-1} + \frac{(1+r)^{\lfrf{n^{1/2}}-1}}{2}
(\ga')^{\lfrf{n^{1/2}}}
\left(V(x) + \E (V(X_0)) + \frac{1}{1-\ga'}\right)
\\
&+
c n^{3/2} (1+r)^{\lfrf{n^{1/2}}-1} (\ga')^{\lfrf{n^{1/2}}-1}
\\
&\le 
\bar{\beta}^{n^{1/2}-1} +
\left[\frac{(V(x) + \E (V(X_0)) + \frac{1}{1-\tilde{\ga}}}{2}+c n^{3/2}\right]
\tilde{\ga}^{\lfrf{n^{1/2}}-1}
\end{align*}
for some $c>0$,
whenever $n\ge (N+2)^2$ which implies the desired estimate.
\end{proof}

The proof of Theorem \ref{thm:stac_forward_coupling} follows a similar approach to that of Lemma \ref{lem:MCREmixing}, with a few key differences. The most significant distinction to keep in mind is that the distribution of $X_0^\ast$ heavily depends on the whole trajectory $(Y_n)_{n\in\Z}$.
\begin{lemma}\label{lem:VYbound}
Under the conditions of Theorem \ref{thm:stac_forward_coupling}, we have
$$
\limsup_{n\to\infty}\E^{1/n} \left[V(X_0^\ast)\prod_{k=0}^{n-1}\ga (Y_k)\right]\le \bar{\ga}.
$$
\end{lemma}
\begin{proof}
	Without the loss of generality, for our convenience, we can assume that the i.i.d. driving noise is a double-sided infinite process $(\eps_n)_{n\in\Z}$, as well. Let $x\in\X$ be arbitrary and deterministic. Then by Corollary 1 and the subsequent Note in  \cite{Truquet1},
	$\law{Z_{0,m}^{x,\mathbf{Y}}}=\law{Z_{-m,0}^{x,\mathbf{Y}}}\to\law{X_0^\ast}$, as $m\to\infty$ in total variation hence
	\begin{equation}\label{eq:origest}
		\E \left[V(X_0^\ast)\prod_{k=0}^{n-1}\ga (Y_k)\right] =
		\lim_{M\to\infty}
		\lim_{m\to\infty}
		\E \left[
		\min\left(
		M,V(Z_{-m,0}^{x,\mathbf{Y}})\prod_{k=0}^{n-1}\ga (Y_k)
		\right)
		\right].
	\end{equation}
	
	By Lemma \ref{lem:ita}, we can write
	\begin{align*}
		\E \left[
		V(Z_{-m,0}^{x,\mathbf{Y}})\prod_{k=0}^{n-1}\ga (Y_k)
		\middle| \F_{-\infty,\infty}^Y
		\right]
		&=[Q(Y_{-1})\ldots Q(Y_{-m})V](x)\prod_{k=0}^{n-1}\ga (Y_k)
		\\
		&\le
		V(x)\,\prod_{r=-m}^{n-1}
		\ga (Y_r)
		+ \sum_{r=-m}^{-1}K(Y_r)\prod_{j=r+1}^{n-1}
		\ga (Y_j)
	\end{align*}
	thus by the tower rule, we have
	\begin{align}\label{eq:prevest}
		\E \left[
		\min\left(
		M,V(Z_{-m,0}^{x,\mathbf{Y}})\prod_{k=0}^{n-1}\ga (Y_k)
		\right)
		\right] \le 
			V(x)\,\E\left[\prod_{r=-m}^{n-1}
		\ga (Y_r)
		\right]
		+ \sum_{r=-m}^{-1}
		\E
		\left[
		K(Y_r)\prod_{j=r+1}^{n-1}
		\ga (Y_j)
		\right].
	\end{align}.
	
	By Assumtion \ref{as:dynamics}, for any fixed $\bar{\ga}<\ga'<1$, exists $N\in\N$ such that 
	$$
	\E\left(K(Y_0)\prod_{j=1}^n\ga (Y_j)\right)\le (\ga')^n,\,\,n\ge N,
	$$
	and thus we can further estimate the right-hand side in \eqref{eq:prevest} as follows:
	\begin{align*}
			\E \left[
		\min\left(
		M,V(Z_{-m,0}^{x,\mathbf{Y}})\prod_{k=0}^{n-1}\ga (Y_k)
		\right)
		\right] 
		&\le V(x)(\ga')^{n+m}+\sum_{r=-m}^{-1} (\ga')^{n-1-r}
		\\
		&\le V(x)(\ga')^{n+m} +\frac{(\ga')^{n}}{1-\ga'}.
	\end{align*}
	Substituting this estimate back into \eqref{eq:origest} yields
	$$
	\limsup_{n\to\infty}\E^{1/n} \left[V(X_0^\ast)\prod_{k=0}^{n-1}\ga (Y_k)\right] \le \ga'.
	$$
	Now, taking the limit $\ga'\downarrow\bar{\ga}$ gives the desired inequality.
\end{proof}

\begin{proof}[Proof of Theorem \ref{thm:stac_forward_coupling}]
We consider the processes $Z_{0,n}^{X_0,\mathbf{Y}}$ and $Z_{0,n}^{X_0^\ast,\mathbf{Y}}$, $n\in\N$, where the latter is a version of $(X_n^\ast)_{n\in\N}$.
For the coupling time, along similar lines as in the proof of Lemma \ref{lem:MCREmixing}, 
by using Lemma \ref{lem:coupling}, we obtain
\begin{align*}
	\P (\tau> n)&\le \P \left(Z_{0,n}^{X_0,\mathbf{Y}}\ne Z_{0,n}^{X_0^\ast,\mathbf{Y}}\right) 
	\\
	&\le
	 \bar{\beta}^m
	+ 
	\frac{(1+r)^{\lfrf{\frac{n-1}{m}}}}{2}
	\E\left[
	\left(
	\E \left[V(X_0^\ast)\mid \F_{-\infty,\infty}^Y\right]+\E (V(X_0))\right)
	\prod_{r=0}^{\lfrf{\frac{n-1}{m}}}\ga (Y_{r})
	\right]
	\\
	&+
	(1+r)^{\lfrf{\frac{n-1}{m}}}
	\sum_{k=0}^{m-1}\sum_{l=k}^{k\lfrf{\frac{n-1}{m}}} \E\left[K(Y_{l})\prod_{r=1}^{\lfrf{\frac{n-1}{m}}}\ga (Y_{l+r})
	\right].
\end{align*}
where $0<m<n$. Again, as in the proof of Lemma \ref{lem:MCREmixing}, let us fix $m_n=\frac{n-1}{\lfrf{n^{1/2}}-1}$, $n>1$, and $\ga'>\bar{\ga}$ such that $\tilde{\ga}:=(1+r)\ga'<1$. 
By Assumption \ref{as:dynamics} and Lemma \ref{lem:VYbound}, we can choose $N\in\N$ be so large that such that
\begin{align*}
	\E\left[K(Y_0)\prod_{k=1}^n \ga (Y_k)\right]\le (\ga')^n
	\,\,\text{and}\,\,
	\E\left[V(X_0^\ast)\prod_{k=0}^{n-1} \ga (Y_k)\right]\le (\ga')^n,\,\,n\ge N.
\end{align*}
For $\lfrf{n^{1/2}}\ge N$, we have
\begin{align*}
	\P (\tau\ge n)
	&\le 
	\bar{\beta}^{\lfrf{n^{1/2}}}+\frac{(1+r)^{\lfrf{n^{1/2}}-1}}{2}
	\left(
	1+\E(V(X_0))
	\right)(\ga')^{\lfrf{n^{1/2}}}\\
	&+
	(1+r)^{\lfrf{n^{1/2}}-1}
	\sum_{k=0}^{m_n-1}\sum_{l=k}^{k(\lfrf{n^{1/2}}-1)}
	(\ga')^{\lfrf{n^{1/2}}-1}
	\\
	&\le
	\bar{\beta}^{\lfrf{n^{1/2}}}
	+
	\frac{1+\E (V(X_0))}{2}\tilde{\ga}^{\lfrf{n^{1/2}}}
	+
	\tilde{\ga}^{\lfrf{n^{1/2}}-1}\sum_{k=0}^{m_n-1}(k(\lfrf{n^{1/2}}-2)+1),
\end{align*}
where $\sum_{k=0}^{m_n-1}(k(\lfrf{n^{1/2}}-2)+1)=O(n^{3/2})$ which completes the proof.   	
\end{proof}

\end{document}